\documentclass[11pt]{amsart}

\usepackage{amsmath}
\usepackage{amsfonts}
\usepackage{amssymb}
\usepackage{amscd}
\usepackage{amsthm}
\usepackage{filecontents}
\usepackage{framed}
\usepackage{fullpage}
\usepackage{graphicx}
\usepackage{latexsym}
\usepackage[numbers]{natbib}  

\usepackage{hyperref}
\hypersetup{colorlinks,linkcolor={red},citecolor={blue},urlcolor={blue}}

\evensidemargin\oddsidemargin

\newtheorem{theorem}{Theorem}[section]
\newtheorem{lemma}[theorem]{Lemma}
\newtheorem{corollary}[theorem]{Corollary}

\theoremstyle{definition}
\newtheorem{definition}[theorem]{Definition}

\newtheorem{remark}[theorem]{Remark}

\numberwithin{equation}{section}

\newcommand{\CC}{\mathbb C}

\newcommand{\cD}{\mathcal D}

\newcommand{\Orth}{\mathop{\null\mathrm {O}}\nolimits}

\newcommand{\latt}[1]{{\langle{#1}\rangle}}

\newcommand{\ord}{\operatorname{ord}}

\usepackage{longtable}

\usepackage{tikz}
\usetikzlibrary{positioning}
\usetikzlibrary{arrows}

\newenvironment{psmallmatrix}
  {\left(\begin{smallmatrix}}
  {\end{smallmatrix}\right)}

\begin{document}

\title[Simple lattices and free algebras of modular forms]{Simple lattices and free algebras of modular forms}

\author{Haowu Wang}

\address{Max-Planck-Institut f\"{u}r Mathematik, Vivatsgasse 7, 53111 Bonn, Germany}

\email{haowu.wangmath@gmail.com}

\author{Brandon Williams}

\address{Fachbereich Mathematik, Technische Universit\"{a}t Darmstadt, 64289 Darmstadt, Germany}

\email{bwilliams@mathematik.tu-darmstadt.de}


\date{\today}


\begin{abstract}
We study the algebras of modular forms on type IV symmetric domains for simple lattices; that is, lattices for which every Heegner divisor occurs as the divisor of a Borcherds product. For every simple lattice $L$ of signature $(n,2)$ with $3\leq n \leq 10$, we prove that the graded algebra of modular forms for the maximal reflection subgroup of the orthogonal group of $L$ is freely generated. We also show that, with five exceptions, the graded algebra of modular forms for the maximal reflection subgroup of the discriminant kernel of $L$ is also freely generated.
\end{abstract}

\maketitle

\begin{small}
\tableofcontents
\end{small}

\addtocontents{toc}{\setcounter{tocdepth}{1}} 

\section{Introduction and statement of results}

An even integral lattice $L$ of signature $(n, 2)$ is called \emph{simple} if the dual Weil representation attached to $L$ admits no cusp forms of weight $1 + n/2$. In particular $L$ is simple if and only if every Heegner divisor on the modular variety attached to $L$ occurs as the divisor of a Borcherds product. There are finitely many simple lattices and they were determined by Bruinier--Ehlen--Freitag \cite{BEF}. In this paper we will prove:

\begin{theorem}\label{Maintheorem}
Let $L$ be a simple lattice of signature $(n, 2)$, with $3 \le n \le 10$. 
\begin{itemize}
\item[(i)] Let $\Orth_r(L)$ denote the subgroup generated by all reflections in the orthogonal group of $L$. Then the graded ring of modular forms $M_*(\Orth_r(L))$ is freely generated.
\item[(ii)]  Let $\widetilde \Orth_r(L)$ denote the subgroup generated by all reflections in the discriminant kernel of $L$. With five exceptions, the graded ring of modular forms $M_*(\widetilde \Orth_r(L))$ is freely generated.
\end{itemize}
\end{theorem}

The main tool in the proof is the necessary and sufficient condition of \cite{Wan} for the graded ring $M_*(\Gamma)$ of modular forms on a type IV symmetric domain for an arithmetic group $\Gamma$ to be free. This asserts that $M_*(\Gamma)$ is freely generated by $n+1$ modular forms if and only if their Jacobian is a cusp form that vanishes precisely on all mirrors of reflections in $\Gamma$ with multiplicity one, where $n$ is the dimension of the symmetric domain. To apply this theorem, we construct the potential Jacobians and the generators using the additive and multiplicative lifts due to Borcherds \cite{Bo}. We can restrict our attention to reflection groups because if $M_*(\Gamma)$ is a free algebra then $\Gamma$ must be generated by reflections (see \cite{VP89}). 

 The five exceptions in part (ii) of Theorem \ref{Maintheorem} are $2U(2)\oplus A_1(2)$, $U\oplus U(2)\oplus A_1(2)$, and $2U\oplus A_1(m)$ for $m=2,3,4$ can be understood in terms of this criterion. For the first of these lattices, $\widetilde{\Orth}_r$ is empty; for the other four, the potential Jacobians fail to be cusp forms. 

Let us explain briefly why one might expect a relationship between free algebras of modular forms and simple lattices. The necessary condition of \cite{Wan} states that the Jacobian $J$ of any generators of a free algebra $M_*(\Gamma)$ is a cusp form for the determinant character which vanishes exactly on all mirrors of reflections in $\Gamma$ with multiplicity one. Moreover, $J^2$ (which has trivial character) factors as the product of modular forms which each vanish precisely on a single $\Gamma$-orbit of these mirrors. In many cases, the converse theorem for Borcherds products \cite{Br} implies that these modular forms must be Borcherds products, and through the ``obstruction principle'' \cite{Bo2}, we see that a \emph{necessary} criterion for $L$ to admit free algebras is that all cusp forms for the dual Weil representation of $L$ of weight $1 + n/2$ have vanishing Fourier coefficients in exponents that correspond to reflections of $L$. By contrast, only in very exceptional cases can the generators themselves be represented as Borcherds products. 

As far as the authors are aware, there are no lattices where this is known to hold except for lattices for which it is trivially true (e.g. simple lattices, or more generally lattices which become simple after dividing out by a group of automorphisms). There are examples of lattices of the latter type that yield free algebras of modular forms. For example, some free algebras associated to the lattices $2U \oplus L$ where $L$ is the $B_n$-root lattice (i.e. $\mathbb{Z}^n$ with the Euclidean bilinear form) appear in \cite{WW20a}, the case $n = 2$ of which is particularly well-known through the interpretation as Hermitian modular forms for the Gaussian integers (cf. \cite{F67}). The methods of this paper can also be applied to this case (see \S \ref{sec:9} for an example). Unfortunately it is not clear how to classify lattices of this type. 

The classification of \cite{BEF} also contains two simple lattices of signature $(n, 2)$ with $n > 10$, namely $\mathrm{II}_{18, 2}$ and $\mathrm{II}_{26, 2}$, but their algebras of modular forms cannot be free by a theorem of Vinberg--Shvartsman \cite{VS17}.  In a future paper we hope to consider the algebras of modular forms for simple lattices of signature $(2, 2)$; we do not carry this out here because the list of such lattices is quite long (according to \cite{BEF} there are $67$ such lattices) and because the failure of Koecher's principle in general must be overcome with different techniques. 

Many of the algebras of modular forms for the 37 lattices covered by Theorem \ref{Maintheorem} appear elsewhere in the literature, and in this paper we simply complete the proof by cases. The literature on algebras of modular forms is quite broad, and (thanks to exceptional isomorphisms in low dimension) many results appear in a rather different form than Theorem \ref{Maintheorem}. We have attributed the known results to the best of our knowledge in Tables \ref{Maintab1} and \ref{Maintab2}, where we also indicate the lattices studied in this paper. Not all of the cases in this paper are new. In particular the structure of $M_*(\widetilde{\Orth}_r(2U(2)\oplus D_4))$ is the main result of \cite{FSM07}, and $M_*(\widetilde \Orth_r(U\oplus U(m)\oplus A_1))$ has been determined in \cite{AI} for $m=2,3,4$. We chose to include these cases because they fit naturally into towers of free algebras and because the Jacobian criterion of \cite{Wan} simplifies the proofs considerably.

\clearpage

\begin{table}[htbp]
\caption{Free algebras of modular forms for the subgroup $\Gamma$ generated by all reflections in $\Orth^+(L)$. $\widetilde \Orth^+(L)$ denotes the discriminant kernel. $\Gamma$ is labeled $\Orth_r(L)$ if it is neither $\Orth^+(L)$ nor $\widetilde \Orth^+(L)$.}\label{Maintab1}
\renewcommand\arraystretch{1.1}
\noindent\[
\begin{array}{cccccc}
\# & L & \Gamma & \text{Weights of generators} & \text{Reference} & \text{Section} \\
\hline
1& 2U \oplus A_1 & \Orth^+ & 4, 6, 10, 12 & \text{\cite{I62}} & \\
2& 2U \oplus A_2 & \Orth^+ & 4, 6, 10, 12, 18 & \text{\cite{DK03}} &\\
3& 2U \oplus A_3 & \Orth^+ & 4, 6, 8, 10, 12, 18& \text{\cite{Klo05}} & \\
4& 2U \oplus A_4 & \widetilde{ \Orth}^+ & 4, 6, 7, 8, 9, 10, 12 & \text{\cite{WW20a}} &\\
5& 2U \oplus A_5 & \widetilde{ \Orth}^+ & 4, 6, 6, 7, 8, 9, 10, 12& \text{\cite{WW20a}} &\\
6& 2U \oplus A_6 & \widetilde{ \Orth}^+ & 4, 5, 6, 6, 7, 8, 9, 10, 12& \text{\cite{WW20a}} &\\
7& 2U \oplus A_7 & \widetilde{ \Orth}^+ & 4, 4, 5, 6, 6, 7, 8, 9, 10, 12& \text{\cite{WW20a}} &\\
8& 2U \oplus D_4 & \Orth^+ & 4, 6, 10, 12, 16, 18, 24 & \text{\cite{K05}} &\\
9& 2U \oplus D_5 & \Orth^+ & 4, 6, 8, 10, 12, 14, 16, 18& \text{\cite{V10}} &\\
10& 2U \oplus D_6 & \Orth^+ & 4, 6, 8, 10, 12, 12, 14, 16, 18& \text{\cite{V18}} &\\
11& 2U \oplus D_7 & \Orth^+ & 4, 6, 8, 10, 10, 12, 12, 14, 16, 18& \text{\cite{V18}} &\\
12& 2U \oplus D_8 & \Orth^+ & 4, 6, 8, 8, 10, 10, 12, 12, 14, 16, 18& \text{\cite{V18}} &\\
13& 2U \oplus E_6  & \widetilde{ \Orth}^+ & 4, 6, 7, 10, 12, 15, 16, 18, 24& \text{\cite{WW20a}} &\\
14& 2U \oplus E_7 & \Orth^+ & 4, 6, 10, 12, 14, 16, 18, 22, 24, 30 & \text{\cite{WW20a}} &\\
15& 2U \oplus E_8 & \Orth^+ & 4, 10, 12, 16, 18, 22, 24, 28, 30, 36, 42 & \text{\cite{HU14}} &\\
\hline
16& 2U \oplus A_1(2)  & \Orth_r & 4, 4, 6, 6 & \text{\cite{F67}} &\\
17& 2U \oplus A_1(3)  & \Orth_r & 2, 4, 4, 6 & \text{\cite{D02}} &\\
18& 2U \oplus A_1(4)  & \Orth_r & 1, 3, 4, 6 & \text{\cite{I62}} &\\
\hline
19& U \oplus U(2) \oplus A_1 & \Orth^+ & 2, 4, 6, 8 & \text{\cite{I91}} &\S\ref{sec:3}\\
20& U \oplus U(2) \oplus A_2 & \Orth^+ & 2, 4, 6, 8, 10 & &\\
21& U \oplus U(2) \oplus A_3 & \Orth^+ & 2, 4, 6, 8, 10, 12& &\\
22& U \oplus U(2) \oplus D_4 & \Orth^+ & 2, 6, 8, 10, 12, 16, 20& &\\
\hline
23& U(2) \oplus S_8 & \Orth^+ & 2, 4, 6, 8 & & \S \ref{sec:4}\\
24& 2U(2) \oplus A_2 & \Orth_r & 4, 4, 6, 6, 6 & &\\
25& 2U(2) \oplus A_3 & \Orth^+ & 4, 6, 6, 8, 10, 12& \text{\cite{FSM07}} &\\
26& 2U(2) \oplus D_4 & \Orth^+ & 4, 6, 10, 12, 16, 18, 24 & \text{\cite{K05}} &\\
\hline
27& U \oplus U(3) \oplus A_1 & \Orth_r & 1, 3, 4, 6& \text{\cite{I91}} & \S \ref{sec:5} \\
28& U \oplus U(3) \oplus A_2 & \Orth_r & 1, 4, 6, 9, 12 & &\\
\hline
29& U \oplus U(4) \oplus A_1 & \Orth_r & 1, 3, 4, 6 &  & \S \ref{sec:6}\\
\hline
30& U \oplus S_8 & \Orth^+ & 2, 4, 6, 10 & & \S \ref{sec:7}\\
\hline
31& 2U(3) \oplus A_1 & \Orth^+ & 2, 4, 4, 6 &  &\S \ref{sec:8}\\
32& U \oplus U(2) \oplus A_1(2) & \Orth^+ & 2, 4, 6, 8 &  & \\
\hline
33& 2U(2) \oplus A_1 & \Orth^+ & 4,4,6,6 &  &\S \ref{sec:8.3} \\
34& 2U(4) \oplus A_1 & \Orth^+ & 4, 6, 10, 12 & &\\
35& 2U(2) \oplus A_1(2) & \Orth^+ & 4,6,10,12 & &\\
36& U(2) \oplus U(4) \oplus A_1 & \Orth^+ & 2, 4, 6, 8 & &\\
37& 2U(3) \oplus A_2 & \Orth^+ & 4, 6, 10, 12, 18 & &\\
\hline
\end{array} 
\]
\end{table}

\begin{table}[htbp]
\caption{Free algebras of modular forms for the subgroup $\Gamma$ generated by all reflections in the discriminant kernel $\widetilde \Orth^+(L)$. $\Gamma$ is labeled $\widetilde \Orth_r(L)$ if it is neither $\Orth^+(L)$ nor $\widetilde \Orth^+(L)$.}
\renewcommand\arraystretch{1.1}\label{Maintab2}
\noindent\[
\begin{array}{cccccc}
\# & L & \Gamma & \text{Weights of generators} & \text{Reference} & \text{Section}\\
\hline
1& 2U \oplus A_1 & \Orth^+ & 4, 6, 10, 12 & \text{\cite{I62}} & \\
2& 2U \oplus A_2 & \widetilde \Orth^+ & 4, 6, 9, 10, 12& \text{\cite{DK03}} & \\
3& 2U \oplus A_3 & \widetilde \Orth^+ & 4, 6, 8, 9, 10, 12& \text{\cite{Klo05}} & \\
4& 2U \oplus A_4 & \widetilde \Orth^+ & 4, 6, 7, 8, 9, 10, 12& \text{\cite{WW20a}} &\\
5& 2U \oplus A_5 & \widetilde \Orth^+ & 4, 6, 6, 7, 8, 9, 10, 12& \text{\cite{WW20a}} & \\
6& 2U \oplus A_6 & \widetilde \Orth^+ & 4, 5, 6, 6, 7, 8, 9, 10, 12& \text{\cite{WW20a}} & \\
7& 2U \oplus A_7 & \widetilde \Orth^+ & 4, 4, 5, 6, 6, 7, 8, 9, 10, 12& \text{\cite{WW20a}} & \\
8& 2U \oplus D_4 & \widetilde \Orth^+ & 4, 6, 8, 8, 10, 12, 18& \text{\cite{WW20a}} & \\
9& 2U \oplus D_5 & \widetilde \Orth^+ & 4, 6, 7, 8, 10, 12, 16, 18& \text{\cite{WW20a}} & \\
10& 2U \oplus D_6 & \widetilde \Orth^+ & 4, 6, 6, 8, 10, 12, 14, 16, 18& \text{\cite{WW20a}} & \\
11& 2U \oplus D_7 & \widetilde \Orth^+ & 4, 5, 6, 8, 10, 12, 12, 14, 16, 18& \text{\cite{WW20a}} & \\
12& 2U \oplus D_8 & \widetilde \Orth^+ & 4, 4, 6, 8, 10, 10, 12, 12, 14, 16, 18& \text{\cite{WW20a}} & \\
13& 2U \oplus E_6 & \widetilde \Orth^+ & 4, 6, 7, 10, 12, 15, 16, 18, 24& \text{\cite{WW20a}} & \\
14& 2U \oplus E_7 & \Orth^+ & 4, 6, 10, 12, 14, 16, 18, 22, 24, 30& \text{\cite{WW20a}} & \\
15& 2U \oplus E_8 & \Orth^+ & 4, 10, 12, 16, 18, 22, 24, 28, 30, 36, 42& \text{\cite{HU14}} & \\
\hline
16& 2U(2) \oplus A_1 & \widetilde \Orth^+ & 2, 2, 2, 2 & \text{\cite{WW20b}} &\\
17& 2U(3) \oplus A_1 & \widetilde \Orth^+ & 1, 1, 1, 1 & \text{\cite{WW20b}} &\\
18& 2U(4) \oplus A_1 & \widetilde \Orth_r  & 1/2, 1/2, 1/2, 1/2 & \text{\cite{R93}} & \\
19& U(2) \oplus U(4) \oplus A_1 & \widetilde \Orth^+ & 1, 1, 1, 1 & \text{\cite{WW20b}} & \\
20& 2U(3) \oplus A_2 & \widetilde \Orth^+ & 1, 1, 1, 1, 1& \text{\cite{FSM06}} & \\
\hline
21& U \oplus U(2) \oplus A_1 & \widetilde \Orth^+ & 2, 4, 4, 6 & \text{\cite{I91}} &\S \ref{sec:3} \\
22& U \oplus U(2) \oplus A_2 & \widetilde \Orth^+ & 2, 4, 4, 5, 6 & &\\
23& U \oplus U(2) \oplus A_3 & \widetilde \Orth^+ & 2, 4, 4, 4, 5, 6 & &\\
24& U \oplus U(2) \oplus D_4 & \widetilde \Orth^+ & 2, 4, 4, 4, 4, 6, 10 & & \\
\hline
25& U(2) \oplus S_8 & \widetilde \Orth^+  & 2, 2, 2, 3 & & \S \ref{sec:4} \\
26& 2U(2) \oplus A_2 & \widetilde \Orth^+ & 2, 2, 2, 2, 3 &\\
27& 2U(2) \oplus A_3 & \widetilde \Orth^+ & 2, 2, 2, 2, 2, 3 & & \\
28& 2U(2) \oplus D_4 & \widetilde \Orth^+ & 2, 2, 2, 2, 2, 2, 6 & \text{\cite{FSM07}} & \\
\hline
29& U \oplus U(3) \oplus A_1 & \widetilde \Orth^+ & 1, 3, 3, 4 & \text{\cite{I91}} &\S \ref{sec:5} \\
30& U \oplus U(3) \oplus A_2 & \widetilde \Orth^+ & 1, 3, 3, 3, 4 &  \\
\hline
31& U \oplus U(4) \oplus A_1 & \widetilde \Orth^+ & 1, 2, 2, 3 & \text{\cite{HI05}}& \S \ref{sec:6} \\
\hline
32& U \oplus S_8 & \widetilde \Orth^+ & 2, 4, 5, 6 & & \S \ref{sec:7} \\
\hline
\end{array} 
\]
\end{table}

\clearpage

\section{Preliminaries}


\subsection{Vector-valued modular forms}

Let $L = (L, Q)$ be an even integral lattice with induced bilinear form $\langle x,y \rangle = Q(x+y) - Q(x) - Q(y)$ and dual lattice $$L' = \{x \in L \otimes \mathbb{Q}: \; \langle x, y \rangle \in \mathbb{Z}\}.$$ There is an induced $\mathbb{Q}/\mathbb{Z}$-valued quadratic form on the discriminant group of $L$: $$Q : L'/L \rightarrow \mathbb{Q}/\mathbb{Z}, \; Q(x + L) = Q(x) + \mathbb{Z}.$$ The pair $A := (L'/L, Q)$ is the \emph{discriminant form} of $L$.

Let $\mathrm{Mp}_2(\mathbb{Z})$ be the metaplectic group, consisting of pairs $M = (M, \phi_M)$ where $M = \begin{psmallmatrix} a & b \\ c & d \end{psmallmatrix} \in \mathrm{SL}_2(\mathbb{Z})$ and $\phi_M$ is a holomorphic square root of $\tau \mapsto c \tau + d$ on $\mathbb{H}$, with the standard generators $T = (\begin{psmallmatrix} 1 & 1 \\ 0 & 1 \end{psmallmatrix}, 1)$ and $S = (\begin{psmallmatrix} 0 & -1 \\ 1 & 0 \end{psmallmatrix}, \sqrt{\tau})$. The \emph{Weil representation} $\rho_L$ is the representation of $\mathrm{Mp}_2(\mathbb{Z})$ on the group ring $\mathbb{C}[A] = \mathrm{span}(e_x: \; x \in A)$ defined by $$\rho_L(T) e_x = \mathbf{e}(-Q(x)) e_x \; \text{and} \; \rho_L(S) e_x = \frac{\mathbf{e}(\mathrm{sig}(L) / 8)}{\sqrt{|A|}} \sum_{y \in A} \mathbf{e}(\langle x,y \rangle) e_y.$$ We remark that $\rho_L$ factors through $\mathrm{SL}_2(\mathbb{Z})$ if and only if $\mathrm{sig}(L)$ is even, and it factors through $\mathrm{PSL}_2(\mathbb{Z})$ if and only if $\mathrm{sig}(L) \equiv 0$ mod $4$.

A \emph{modular form} of weight $k \in \frac{1}{2}\mathbb{Z}$ for the Weil representation $\rho_L$ is a holomorphic function $f : \mathbb{H} \rightarrow \mathbb{C}[A]$ that satisfies $$f|_k M(\tau)  := \phi_M(\tau)^{-2k} f(M \cdot \tau) = \rho_L(M) f(\tau)$$ for all $M \in \mathrm{Mp}_2(\mathbb{Z})$, and which is holomorphic in infinity. $f$ is represented by a Fourier series: $$f(\tau) = \sum_{x \in A} \sum_{n \in \mathbb{Z} - Q(x)} c(n, x) q^n e_x,$$ and holomorphy in infinity is equivalent to $c(n, x) = 0$ for all $n < 0$.

The central element $$Z = S^2 = (ST)^3 = \left( \begin{psmallmatrix} -1 & 0 \\ 0 & -1 \end{psmallmatrix}, i \right)$$ acts through the Weil representation by $\rho_L(Z) e_x = \mathbf{e}(\mathrm{sig}(L) / 4) e_{-x}$ and acts trivially on $\mathbb{H}$. Comparing $f|_k Z$ with $f$ shows that nonzero modular forms exist only in weights $k$ for which $\kappa := k + \mathrm{sig}(L)/2$ is integral, and in this case their Fourier coefficients satisfy $c(n, x) = (-1)^{\kappa} c(n, -x)$.

The simplest vector-valued modular forms are Eisenstein series (cf. chapter 1 of \cite{Br}). If $k \ge 5/2$ and $\beta \in A$ with $Q(\beta) = 0$ then the form $$E_{k, \beta}(\tau) = \sum_{M \in \Gamma_{\infty} \backslash \mathrm{Mp}_2(\mathbb{Z})} e_{\beta} |_k M$$ converges absolutely and locally uniformly and defines a modular form of weight $k$ whose Fourier expansion takes the form $$E_{k,\beta}(\tau) = \frac{1}{2} (e_{\beta} + (-1)^{\kappa} e_{-\beta}) + \sum_{x \in L'/L} \sum_{\substack{n \in \mathbb{Z} - Q(x) \\ n > 0}} c(n, x) q^n e_x.$$ Here $\Gamma_{\infty} = \langle T, Z \rangle = \{M \in \mathrm{Mp}_2(\mathbb{Z}): \; M \cdot \infty = \infty\}$ and $e_{\beta}$ is understood as a constant function. The space of modular forms decomposes in the form $$M_k(\rho_L) = S_k(\rho_L) \oplus \bigoplus_{\substack{\beta \in A / \pm \\ Q(\beta) = 0}} E_{k, \beta}$$ where $S_k(\rho_L)$ is the space of cusp forms. The Fourier coefficients of all $E_{k, \beta}$ are rational numbers. Following \cite{Wil}, \cite{Wil2}, when $\kappa$ is even the series $$Q_{k, m, \beta}(\tau) = \sum_{M \in \Gamma_{\infty} \backslash \mathrm{Mp}_2(\mathbb{Z})} \Big( \sum_{n=1}^{\infty} q^{mn^2} e_{m \beta} \Big) \Big|_k M, \; \beta \in A, \; m \in \mathbb{Z} - Q(\beta)$$ yield a spanning set of vector-valued cusp forms with easily computed rational Fourier coefficients; and similarly when $\kappa$ is odd and $k \ge 7/2$, the forms $$R_{k, m, \beta}(\tau) = \sum_{M \in \Gamma_{\infty} \backslash \mathrm{Mp}_2(\mathbb{Z})} \Big( \sum_{n=1}^{\infty} n q^{mn^2} e_{m \beta} \Big) \Big|_k M, \; \beta \in A, \; m \in \mathbb{Z} - Q(\beta)$$ yield a spanning set of cusp forms. The spaces of modular forms in weights $k \le 2$ can be determined by intersecting $$M_k(\rho_L) = E_4(\tau)^{-1} M_{k+4}(\rho_L) \cap E_6(\tau)^{-1} M_{k+6}(\rho_L),$$ where $E_4, E_6$ are the usual (scalar) Eisenstein series.

The integral orthogonal group $\Orth(L)$ acts on $\mathbb{C}[A]$ by $g \cdot e_x = e_{g \cdot x}$, and this induces an action on modular forms $f$ by $(g \cdot f)(\tau) = g \cdot (f(\tau))$. The action is simple to describe in terms of the modular forms of the previous paragraph: $$g \cdot E_{k, \beta} = E_{k, g \cdot \beta}, \; g \cdot Q_{k, m, \beta} = Q_{k, m, g \cdot \beta}, \; g \cdot R_{k, m, \beta} = R_{k, m, g \cdot \beta}.$$ (Indeed this holds for arbitrary automorphisms of the discriminant form, not all of which are induced by elements of $\Orth(L)$.) One can use this observation to determine the spaces of modular forms with any desired behavior under symmetries of the discriminant form. 

The dimensions of the spaces of modular forms can be computed effectively in weights $k \ge 2$ by the Riemann-Roch formula. We will also mention that the module of vector-valued modular forms for any fixed lattice over the graded ring $M_*(\mathrm{SL}_2(\mathbb{Z})) = \mathbb{C}[E_4, E_6]$ is free and finitely-generated \cite{MM} so the dimensions can be expressed conveniently as a Hilbert series: $$\mathrm{Hilb}_{\rho}(t) := \sum_{k \in \frac{1}{2}\mathbb{Z}} \mathrm{dim}\, M_k(\rho) t^k =  t^{\varepsilon / 2} \frac{P_{\rho}(t)}{(1 - t^4)(1 - t^6)}, \; \varepsilon := \mathrm{sig}(L) \, \text{mod} \, 2$$ for some polynomial $P_{\rho} \in \mathbb{Z}[t]$ which satisfies $P_\rho(1)=|A|$.

\subsection{Orthogonal modular forms}

For background on modular forms and varieties associated to orthogonal groups we refer to \cite{FH00}. In this section suppose the lattice $L$ has signature $(n, 2)$ for some $n \in \mathbb{N}$. The Hermitian symmetric domain $\mathcal{D}(L)$ is either of the two connected components of $$\{[\mathcal{Z}] \in \mathbb{P}(L \otimes \mathbb{C}): \; \langle \mathcal{Z}, \mathcal{Z} \rangle = 0, \; \langle \mathcal{Z}, \overline{\mathcal{Z}} \rangle < 0 \}.$$

Let $\Orth^+(L)$ denote the subgroup of $\Orth(L)$ that preserves $\mathcal{D}(L)$ and let $\Gamma \le \Orth^+(L)$ be a finite-index subgroup. The kernel of $\Orth^+(L)$ on $\CC[L'/L]$ is called the discriminant kernel of $L$, denoted $\widetilde{\Orth}^+(L)$. Following Baily--Borel \cite{BB66}, the modular variety $X_{\Gamma}$ is constructed from the quotient $Y_{\Gamma} = \Gamma \backslash \mathcal{D}(L)$ by including finitely many zero- and one-dimensional cusps, which correspond to isotropic lines and planes of $L$ that lie in the closure of $\mathcal{D}(L)$ up to equivalence.

\begin{definition} Let $k \in \mathbb{N}_0$. A \emph{modular form} of weight $k$ for a character $\chi : \Gamma \rightarrow \mathbb{C}^{\times}$ is a holomorphic function $F : \mathcal{A}(L) \rightarrow \mathbb{C}$ on the affine cone $$\mathcal{A}(L) := \{\mathcal{Z} \in (L \otimes \mathbb{C}) \backslash \{0\}: \; [\mathcal{Z}] \in \mathcal{D}(L)\}$$ that satisfies $$F(t \mathcal{Z}) = t^{-k} F(\mathcal{Z}) \; \text{for all} \; t \in \mathbb{C}^{\times}$$ and $$F(g \mathcal{Z}) = \chi(g) F(\mathcal{Z}) \; \text{for all} \; g \in \Gamma$$ and which extends holomorphically to all cusps of $X_{\Gamma}$.
\end{definition}

Modular forms can be expanded into Fourier series on the tube domain around any cusp. We describe this in the simpler situation that $L$ is split by a hyperbolic plane; i.e. $L = U(N) \oplus K$ for some $N \in \mathbb{N}$. Write elements of $L$ in the form $w = (\lambda, v, \mu)$ with $v \in K$ and $\lambda,\mu \in \mathbb{Z}$, such that the quadratic form on $L$ is $Q(w) = N\lambda \mu + Q(v)$. In this case the tube domain about the cusp $[(1, 0, 0)]$ is $$\mathbb{H}_K := \text{a connected component of} \; \{z = x+iy \in K \otimes \mathbb{C}: \; Q(y) > 0\}.$$ With the correct choice of connected component, there is a biholomorphic embedding $$\phi : \mathbb{H}_K \rightarrow \mathcal{A}(L), \; \phi(z) = (1, z, -Q(z)).$$ By abuse of notation, if $F$ is a modular form then we also denote by $F$ the function $F(z) := F(\phi(z))$ on $\mathbb{H}_K$.

The group $\Gamma$ acts on $\mathbb{H}_K$ by $$g \cdot z = w \; \text{if and only if} \; g \phi(z) = j(g; z) \phi(w) \; \text{for some} \; j(g, z) \in \mathbb{C}^{\times}.$$ Thus the automorphy equations $F(t \mathcal{Z}) = t^{-k} F(\mathcal{Z})$ and $F(g \mathcal{Z}) = \chi(g) F(\mathcal{Z})$ become the usual functional equation $$F(g \cdot z) = \chi(g) j(g; z)^k F(z), \; g \in \Gamma.$$ The group $\Orth^+(L)$ contains the maps $$t_b : (\lambda, v, \mu) \mapsto (\lambda, v + b, \mu + \langle v, b \rangle - Q(b) \lambda), \; \; b \in K$$ which act on $\mathbb{H}_K$ by $t_b \cdot z = z + b$, so the invariance of $F$ under translations implies that it is represented by a Fourier series: $$F(z) = \sum_{\lambda \in K_1} c(\lambda) \mathbf{q}^{\lambda}, \; \mathbf{q}^{\lambda} := e^{2\pi i \langle \lambda, z \rangle},\; \text{where} \; K_1 = \{b \in K \otimes \mathbb{Q}: \; t_b \in \Gamma \cap \ker(\chi)\}.$$

\subsection{Theta lift and Borcherds products}

We continue to assume that $L$ is a signature $(n, 2)$ lattice that splits in the form $L = U(N) \oplus K$ with tube domain $\mathbb{H}_K$ for some $N \in \mathbb{N}$, and write vectors of $L$ in the form $(a, v, b)$ with $a, b \in \mathbb{Z}$ and $v \in K$. The choice of connected component in the definition of $\mathbb{H}_K$ fixes a \emph{positive cone} for $K$: $$C := \{\mathrm{im}(z): \; z \in \mathbb{H}_K\} \subseteq K \otimes \mathbb{R}.$$

\begin{definition} Let $r \in L'$ be a vector of positive norm. The \emph{rational quadratic divisor} associated to $r$ is $$\mathcal{D}_r(L) = r^{\perp} \cap \mathcal{D}(L) = \{[\mathcal{Z}] \in \mathcal{D}(L): \; \langle \mathcal{Z}, r \rangle = 0\}.$$
\end{definition}

For any $\beta \in L'/L$ and $m \in \mathbb{Z} + Q(\beta)$, $m > 0$, one defines the \emph{Heegner divisor} of index $(m, \beta)$ as the union $$H(m, \beta) = \bigcup_{\substack{r \in L + \beta \\ Q(r) = m}} \mathcal{D}_r(L).$$ We also denote by $H(m, \beta)$ the preimage under the map $\phi : \mathbb{H}_K \rightarrow \mathcal{D}(L)$, $z \mapsto [1:z:-Q(z)]$. Note that with this definition there are inclusions $H(n^2 m, n \beta) \subseteq H(m, \beta)$ for all $n \in \mathbb{N}$.

\begin{theorem} Let $k \in \mathbb{N}$, $k \ge 2$ and define $\ell := k + 1 - n/2$. Let $$f(\tau) = \sum_{x \in L'/L} \sum_{n \in \mathbb{Z} - Q(x)} c(n, x)q^n e_x \in M_{\ell}^!(\rho_L)$$ be a nearly-holomorphic vector-valued modular form of weight $\ell$. Then the \emph{theta lift} \begin{align*} \Phi_f(z) &= -\frac{N^{k-1}}{2k} \sum_{a, b \in \mathbb{Z}/N\mathbb{Z}} e^{2\pi i ab / N} B_k(a/N) c(0, (b/N, 0, 0)) \\ &+ \sum_{\lambda \in K' \cap C}\sum_{j \in \mathbb{Z}/N\mathbb{Z}}  c(Q(\lambda), (j / N, \lambda, 0)) \sum_{n=1}^{\infty} e^{2\pi i jn / N}n^{k-1} \mathbf{q}^{n \lambda} \end{align*} is a meromorphic orthogonal modular form of weight $k$ without character on the discriminant kernel $\widetilde \Orth^+(L)$. $\Phi_f$ has a pole of order $k$ on every rational quadratic divisor $\mathcal{D}_r(L)$ with $c(-Q(r), r) \ne 0$ and is holomorphic elsewhere.
\end{theorem}
Here $B_k$ is the $k$th Bernoulli polynomial.
\begin{proof} Theorem 14.4 of \cite{Bo}
\end{proof}

\begin{remark} The additive theta lift can also be applied when $k = 1$ but in this case the constant coefficient of $\Phi_f(z)$ must be corrected, as in \cite{Bo}. The formula of \cite{Bo} for the constant coefficient is rather complicated, and in practice (at least when the input form $f$ is holomorphic) it is easier to determine it from the constant terms of the elliptic modular forms $\Phi_f(\mu \tau) \in M_2(\Gamma_0(Q(\mu)))$, where $\mu \in K$ can be any vector of positive norm. Note that cusp forms of weight $2 - n/2$ can lift to non-cusp forms of weight $1$.
\end{remark}

We will often take additive theta lifts as generators for rings of modular forms, and the following criterion due to Bruinier for the injectivity of the theta lift is helpful:

\begin{theorem} Suppose $N = 1$; i.e. $L$ splits as $L = U \oplus K$, and that $K$ is isotropic. Then the theta lift on holomorphic forms is injective in every weight.
\end{theorem}
\begin{proof} This follows from Theorem 1.2, Corollary 1.3 and Theorem 4.2 of \cite{Br2}. 
\end{proof}

The spaces of additive lifts of vector-valued modular forms invariant under $\Gamma$ are called the \emph{Maass subspace}, denoted $\mathrm{Maass}_k(\Gamma)$.

The orthogonal Eisenstein series can be constructed through the theta lift:

\begin{definition} The \emph{Eisenstein series} $\mathcal{E}_k$ of weight $k \ge (n + 3)/2$, $k \in 2 \mathbb{Z}$ is the theta lift of the vector-valued Eisenstein series $E_{k+1-n/2, 0}(\tau)$.
\end{definition}
Note that our definition does not agree exactly with the definition of the Eisenstein series by averaging a constant function over a parabolic subgroup, i.e. the value at $s=0$ in $$E_k(z,s) = \sum_{g \in \Gamma_{\infty} \backslash \Orth^+(L)} j(g;z)^{-k} |j(g;z)|^{-2s} Q(y)^s, \; \; z = x + iy$$ where $\Gamma_{\infty} = \{g \in \Orth^+(L): \; j(g; z) = 1\}$. Instead, $E_k(z, s)$ is essentially the lift of the sum of all vector-valued Eisenstein series associated to the isotropic cosets in $L'/L$: $$E_k(z, s) = \frac{\pi^s (2\pi i)^k}{\Gamma(s+k) \zeta(2s+k)} \Phi_f(z), \; \; f = \sum_{\substack{\beta \in L'/L \\ Q(\beta) = 0}} E_{k + 1 - n/2, \beta}(\tau, s), \; \; \mathrm{Re}[s] \gg 0.$$ In either definition the constant term of the Eisenstein series is $$\frac{(2\pi i)^k}{\Gamma(k) \zeta(k)} = -\frac{2k}{B_k}$$ by convention, where $B_k$ is the $k$th Bernoulli number.

Finally we review Borcherds products. The Borcherds lift is a multiplicative map that takes nearly-holomorphic vector-valued modular forms to orthogonal modular forms, all of whose zeros and poles lie on rational quadratic divisors, and which are represented locally by infinite products. It is closely related to the case $k = 0$ of the (additive) theta lift.

\begin{theorem} Let $f(\tau) = \sum_{x \in L'/L} \sum_{n \in \mathbb{Z} - Q(x)} c(n, x) q^n e_x \in M_{1 - n/2}^!(\rho_L)$ be a nearly holomorphic modular form of weight $1 - n/2$ for which $c(n, x)$ is an integer for all $n \le 0$. There is a function $\Psi_f(z)$ on $\mathbb{H}_K$ with the following properties:
\begin{enumerate}
\item  $\Psi_f$ is a meromorphic modular form of weight $c(0, 0) / 2$. 
\item  The divisor of $\Psi_f$ on $\mathbb{H}_K$ is $$\mathrm{div}\, \Psi_f = \sum_{\beta \in L'/L} \sum_{\substack{m \in \mathbb{Z} + Q(\beta) \\ m > 0}} c(-m, \beta) H(m, \beta).$$
\item  On any connected component $W$ of $\mathbb{H}_K \backslash \bigcup_{c(-m, \beta) \ne 0} H(m, \beta)$ (i.e. a \emph{Weyl chamber}), $\Psi_f$ has the product representation $$\Psi_f(z) = C \cdot \mathbf{q}^{\rho} \prod_{\substack{\lambda \in K' \\ \langle \lambda, W \rangle < 0}} \prod_{j \in \mathbb{Z}/N\mathbb{Z}} \Big( 1 - e^{2\pi i j / N} \mathbf{q}^{\lambda} \Big)^{c(-Q(\lambda), (j/N, \lambda, 0))}$$ for some vector $\rho \in K \otimes \mathbb{Q}$ (called the \emph{Weyl vector} of $\Psi_f$ on $W$) and some constant $C$ of absolute value $N$.
\end{enumerate}
\end{theorem}
Here, $\langle \lambda, W \rangle < 0$ means that $\langle \lambda, w \rangle < 0$ for every $w \in W$. When $L$ is a lattice of the form $U(N_1) \oplus U(N_2) \oplus L_0$ with $N_1, N_2 \in \mathbb{N}$ and $L_0$ positive-definite, one can compute the Weyl vector using Theorem 10.4 of \cite{Bo}. This applies to all but two of the lattices considered in this paper. In the remaining cases, we compute the Weyl vectors recursively using the method of quasi-pullbacks (briefly described in the next paragraphs).

In general, any isometric embedding $\phi : L_1 \rightarrow L_2$ of lattices of signature $(\ell_1, 2)$ and $(\ell_2, 2)$ induces a map on symmetric spaces $$\phi : \mathcal{D}(L_1) \rightarrow \mathcal{D}(L_2), \; [\mathcal{Z}] \mapsto [\phi(\mathcal{Z})]$$ and also on affine cones. The \emph{pullback} along $\phi$, $$\phi^* : F \mapsto F \circ \phi$$ sends modular forms on $L_2$ to modular forms on $L_1$ of the same weight.

Suppose $L_1$ and $L_2$ are of the form $L_1 = U(N) \oplus K_1$ and $L_2 = U(N) \oplus K_2$ and that $\phi$ arises from an embedding $K_1 \rightarrow K_2$ by acting trivially on the $U(N)$ component. Then the pullbacks of a theta lift $F$ on $L_2$ can be computed as theta lifts on $L_1$, and therefore do not require the form $F$ to be computed at all. By applying this idea carefully, it is sometimes possible to reduce questions about generators and relations for modular forms on $L_2$ to modular forms on $L_1$.

In the simplest case, suppose $\phi(K_1)$ is an orthogonal direct summand in $K_2$, i.e. $K_2 = \phi(K_1) \oplus M$ where $M$ is necessarily positive-definite. The basic vector-valued modular forms for the Weil representation for $K_2$ are tensors of the form $f(\tau) \otimes g(\tau)$ where $f \in M_*(\rho_{\phi(K_1)}) \cong M_*(\rho_{K_1})$ and $g \in M_*(\rho_M)$. In this case one can show that the pullback of the lift of $f \otimes g$ is given by $$\phi^* \Phi_{f \otimes g} = \Phi_{f \cdot \langle g, \Theta_M \rangle},$$ where $\langle g, \Theta_M \rangle$ is the scalar-valued modular form $$\langle g, \Theta_M \rangle = \sum_{x \in M'/M} g_x(\tau) \theta_{M, x}(\tau) = \sum_{x \in M'} g_{x + M}(\tau) q^{Q(x)} \in M_*(\mathrm{SL}_2(\mathbb{Z}))$$ if $g(\tau) = \sum_{x \in M'/M} g_x(\tau) e_x$ and if $M$ has theta function $$\Theta_M(\tau) = \sum_{x \in M'} q^{Q(x)} e_x = \sum_{x \in M'/M} \theta_{M, x}(\tau) e_x.$$

In the general case, one can always pass from $K_2$ to a finite-index sublattice in which $\phi(K_1)$ does split as a direct summand by using the down- and up-arrow operators of \cite{Br}. Then the result is that $\phi^* \Phi_f$ is the additive lift of the \emph{theta-contraction} of $f$ along $\phi$. A similar statement holds for Borcherds products and the \emph{quasi-pullback}. We refer to \cite{Ma} for details.

\subsection{The modular Jacobian}

The modular Jacobian is the main tool in our approach to free algebras of modular forms. This was introduced in \cite{AI} for lattices of signature $(3, 2)$. We continue to assume that $L$ is of signature $(n, 2)$ and splits as $U(N) \oplus K$, and we fix coordinates $z_1,...,z_n$ on the tube domain $\mathbb{H}_K$. For a modular form $F : \mathbb{H}_K \rightarrow \mathbb{C}$ let $\nabla F$ denote the gradient $(\partial_{z_1} F,...,\partial_{z_n} F)^T$.

We first recall the definition of reflections. The reflection fixing the rational quadratic divisor $\cD_r(L)$ is defined as
\begin{equation*}
\sigma_r(x)=x-\frac{2\latt{r,x}}{\latt{r,r}}r,  \quad x\in L.
\end{equation*}
The hyperplane $\cD_r(M)$ is called the \emph{mirror} of $\sigma_r$. For a non-zero vector $r\in L'$ we denote its order in $L'/L$ by $\ord(r)$.
For any primitive vector $r\in L'$ of positive norm, $\sigma_r\in\Orth^+(L)$ if and only if there exists a positive integer $d$ such that $\latt{r,r}=\frac{2}{d}$ and $\ord(r)=d$ or $\frac{d}{2}$.  We remark that for any primitive vector $r\in L$  the reflection $\sigma_r$ belongs to $\widetilde{\Orth}^+(L)$ if and only if $\latt{r,r}=2$.

\begin{theorem}\label{th:j} Let $F_1,...,F_{n+1}$ be orthogonal modular forms of weights $k_1,...,k_{n+1}$ for a finite-index subgroup $\Gamma < \Orth^+(L)$ and define $$J := J(F_1,...,F_{n+1}) = \mathrm{det} \begin{pmatrix} k_1 F_1 & ... & k_{n+1} F_{n+1} \\ \nabla F_1 & ... & \nabla F_{n+1} \end{pmatrix}.$$ Then: 
\begin{enumerate}
\item $J$ is a cusp form of weight $n + \sum_{i=1}^{n+1} k_i$ for $\Gamma$ with the determinant character $\mathrm{det}$. 
\item $J \ne 0$ if and only if $F_1,...,F_{n+1}$ are algebraically independent. 
\item Let $r \in L'$ and suppose $\Gamma$ contains the reflection $\sigma_r$. Then $J$ vanishes on the rational quadratic divisor $\mathcal{D}_r(L)$. 
\end{enumerate}
Now suppose $M_*(\Gamma)$ is a free algebra and $F_1,...,F_{n+1}$ are generators. Then $J$ satisfies the following additional properties:  
\begin{enumerate}
\item[(4)] The divisor of $J$ consists exactly of simple zeros on the mirrors of reflections in $\Gamma$. In particular, $J$ is a reflective cusp form. 
\item[(5)] If $\{\Gamma \pi_1,...,\Gamma \pi_s\}$ denote the $\Gamma$-equivalence classes of mirrors of reflections in $\Gamma$, then for each $1 \le i \le s$ there exists a modular form $J_i$ for $\Gamma$ with trivial character and with divisor $\mathrm{div}(J_i) = 2 \Gamma \pi_i$, and the irreducible factorization of $J^2$ in $M_*(\Gamma)$ is $$J^2 = \prod_{i=1}^s J_i.$$
\end{enumerate}
\end{theorem}
\begin{proof} This is proved in Theorem 2.5 and Theorem 3.5 of \cite{Wan}.
\end{proof}

The following converse to Theorem \ref{th:j} provides a sufficient criterion for a graded ring of modular forms to be free:

\begin{theorem}\label{th:j2} Let $\Gamma < \Orth^+(L)$ be a finite-index subgroup and suppose there are modular forms $F_1,...,F_{n+1}$ with trivial character on $\Gamma$ whose Jacobian $$J = J(F_1,...,F_{n+1})$$ vanishes exactly on the mirrors of reflections in $\Gamma$ with multiplicity one. Then $M_*(\Gamma)$ is freely generated by $F_1,...,F_{n+1}$, and $\Gamma$ is generated by reflections whose mirrors lie in the divisor of $J$.
\end{theorem}
\begin{proof} This is proved in Theorem 5.1 of \cite{Wan}.
\end{proof}

\section{The \texorpdfstring{$U\oplus U(2)\oplus R$}{} tower}\label{sec:3}

In this section we will compute the algebras of modular forms for the lattices $U \oplus U(2) \oplus R$, where $R$ belongs to the tower of root lattices $$A_1 \subseteq A_2 \subseteq A_3 \subseteq D_4.$$ We compute the algebra of modular forms for $U \oplus U(2) \oplus A_1$ by showing that the Jacobian of a set of potential generators is a Borcherds product with the appropriate divisor, and we compute the algebras of modular forms for $U \oplus U(2) \oplus R$ for the larger-rank root lattices $R$ using the method of pullbacks.

\subsection{Modular forms on \texorpdfstring{$U\oplus U(2)\oplus A_1$}{}}

The dimensions of spaces of modular forms for the Weil representation $\rho$ associated to the lattice $L = U \oplus U(2) \oplus A_1$ have the generating function $$\sum_{k = 0}^{\infty} \mathrm{dim}\, M_{k + 3/2}(\rho) t^k = \frac{(1 + t^2)^3}{(1 - t^4)(1 - t^6)}.$$ In particular $\mathrm{dim}\, M_{3/2}(\rho) = 1$; and $\mathrm{dim}\, M_{7/2}(\rho) = 3$; and $\mathrm{dim}\, M_{11/2}(\rho) = 4$. Since $L$ splits a unimodular plane over $\mathbb{Z}$, the additive lift is injective and the spaces of additive lifts of weights $2$, $4$ and $6$ have dimension $1$, $3$ and $4$ respectively. We fix a nonzero additive lift of weight $2$ denoted $m_{2,A_1}$. The lattice $L$ admits Borcherds products whose input forms' principal parts with respect to the Gram matrix $$\begin{psmallmatrix} 0 & 0 & 0 & 0 & 1 \\ 0 & 0 & 0 & 2 & 0 \\ 0 & 0 & 2 & 0 & 0 \\ 0 & 2 & 0 & 0 & 0 \\ 1 & 0 & 0 & 0 & 0 \end{psmallmatrix}$$ are given in Table \ref{tab:U+U2+A1} below.

\begin{table}[htbp]
\centering
\caption{Some Borcherds products for $U \oplus U(2) \oplus A_1$}
\label{tab:U+U2+A1}
\begin{tabular}{l*{3}{c}r}
\hline
Name & Weight & Principal part\\
\hline
$b_{1, A_1}$ & $2$ & $4 e_0 + q^{-1/4} e_{(0, 0, 1/2, 1/2, 0)}$ \\
$b_{2, A_1}$ & $2$ & $4 e_0 + q^{-1/4} e_{(0, 1/2, 1/2, 0, 0)}$ \\
$b_{3, A_1}$ & $3$ & $6 e_0 + q^{-1/4} e_{(0, 0, 1/2, 0, 0)}$ \\
$b_{4, A_1}$ & $4$ & $8 e_0 + q^{-1/2} e_{(0, 1/2, 0, 1/2, 0)}$\\
$\Phi_{19,A_1}$ & $19$ & $38 e_0 + q^{-1} e_0$ \\
\hline
\end{tabular}    
\end{table}

In the rest of this subsection we will abbreviate $b_{j, A_1}$ by $b_j$.

Following the notations in the introduction, $\widetilde \Orth_r(L)$ is the subgroup of $\Orth^+(U \oplus U(2) \oplus A_1)$ generated by the reflections associated to the divisor of $\Phi_{19,A_1}$ (i.e. all reflections in the discriminant kernel). Each of $b_1, b_2, b_3$ has a character of order two under $\widetilde \Orth_r(L)$, since they have only simple zeros along a divisor that is preserved under those reflections.

\begin{lemma} The Jacobian $J = J(m_{2}, b_1^2, b_2^2, b_3^2)$ equals $\Phi_{19,A_1}$ up to a nonzero constant multiple.
\end{lemma}
\begin{proof}
By Theorem \ref{th:j} (3), the function $J/\Phi_{19,A_1}$ defines a holomorphic modular form of weight zero and thus it is a constant. A direct calculation shows that the constant is nonzero.
\end{proof}

\begin{theorem}\label{th:U+U2+A1_d}
\noindent
\begin{itemize}
\item[(i)] $\widetilde \Orth_r(L)$ coincides with the discriminant kernel $\widetilde \Orth^+(L)$. 
\item[(ii)] The algebra of modular forms for the discriminant kernel is freely generated: 
$$M_*(\widetilde \Orth^+(L)) = \mathbb{C}[m_2, b_1^2, b_2^2, b_3^2].$$
\end{itemize}
\end{theorem}
\begin{proof} Applying the Jacobian criterion shows that $$M_*(\widetilde \Orth_r(L)) = \mathbb{C}[m_2, b_1^2, b_2^2, b_3^2].$$ Since the additive lifts of weight $4$ and $6$ are contained in $M_*(\widetilde \Orth_r(L))$, comparing dimensions shows that $M_k(\widetilde \Orth_r(L)) = \mathrm{Maass}_k(\widetilde \Orth_r(L))$ for $k \in \{4, 6\}$. In particular, $b_1^2, b_2^2, b_3^2$ are additive lifts and are therefore modular under the entire discriminant kernel without character.
\end{proof}

\begin{remark} The product $b_4$ is modular without character on the discriminant kernel. By comparing Fourier expansions one finds $b_4 = b_1^2 - b_2^2 = (b_1 - b_2) (b_1 + b_2)$. Despite this decomposition, the divisor of $b_4$ is irreducible (because $b_1$ and $b_2$ do not have the same character, so their sum and difference are not modular forms on the discriminant kernel for any character).
\end{remark}

\begin{remark} Using Theorem \ref{th:U+U2+A1_d} and comparing the first coefficients in Fourier series, one can show that up to multiples the (unique) cuspidal  lift of weight $6$ is $b_3^2$; in particular this is simultaneously an additive and a multiplicative lift.
\end{remark}

The ring of modular forms for $\Orth^+(L)$ can be computed by a similar argument. All reflections in $\Orth^+(L)$ are associated to the divisor of $b_4 \Phi_{19}$, so this is the prospective Jacobian of the generators.

\begin{theorem}\label{th:U+U(2)+A1}
\noindent
\begin{itemize} 
\item[(i)] $\Orth^+(L) = \Orth_r(L)$ is generated by reflections. 
\item[(ii)] The algebra of modular forms for $\Orth^+(L)$ is freely generated: $$M_*(\Orth^+(L)) = \mathbb{C}[m_2, \mathcal{E}_4, \mathcal{E}_6, \mathcal{E}_8],$$ where $\mathcal{E}_k$ is the Eisenstein series of weight $k$. The Jacobian of the generators equals $b_4 \Phi_{19}$ up to a constant multiple.
\end{itemize}
\end{theorem}
\begin{proof} By construction the Eisenstein series $\mathcal{E}_k$ are modular without character on $\Orth^+(L)$. The form $m_2$ is the additive theta lift of the weight $3/2$ modular form \begin{align*} f(\tau) &= (1 + 6q + 12q^2 + 8q^3 + ...) (e_0 - e_{(0,0,0,1/2,0)} - e_{(0,1/2,0,0,0)}) \\ &+ (8 q^{3/4} + 24q^{11/4} + ...) (e_{(0,0,1/2,0,0)} - e_{(0,0,1/2,1/2,0)} - e_{(0,1/2,1/2,0,0)}) \\ &+ (-12q^{1/2} - 24q^{3/2} - 24q^{5/2} - ...) e_{(0,1/2,0,1/2,0)} \\ &+ (-6q^{1/4} - 24q^{5/4} - 30q^{9/4} - ...) e_{(0,1/2,1/2,1/2,0)},\end{align*} which is invariant under all automorphisms of $(L'/L, Q)$, so $m_2$ is also modular without character on $\Orth^+(L)$. (Another way to see this is as follows. Since $f$ has rational Fourier coefficients and $\dim M_{3/2}(\rho)=1$, it is invariant under $\Orth(L'/L)$ up to a character of order at most two. This character must be trivial because the $e_0$ component of $f$ is nonzero.)  By expressing these generators in terms of those of Theorem \ref{th:U+U2+A1_d} (or by computing their Jacobian directly) one sees that they are algebraically independent. Their Jacobian $J$ is nonzero and divisible by $J_0 := b_4 \Phi_{19}$ by Theorem \ref{th:j} and both $J$ and $J_0$ have weight $23$, so they are equal up to a nonzero constant multiple.
\end{proof}

\begin{remark}
One can also prove Theorem \ref{th:U+U(2)+A1} more indirectly using the following argument. Since $\dim M_{3/2}(\rho)$ is one-dimensional, $m_2$ is modular on $\Orth^+(L)$ with a character of some (finite) order $a$. If $m_2$ had a nontrivial character on the reflection group $\Orth_r(L)$ then it would have a zero on a mirror of some reflection and be divisible by one of the products in Table \ref{tab:U+U2+A1}, violating Koecher's principle. Then the Jacobian criterion implies $$M_*(\Orth_r(L)) = \CC[m_2, \mathcal{E}_4, \mathcal{E}_6, \mathcal{E}_8],$$ and therefore $$M_*(\Orth^+(L)) = \mathbb{C}[m_2^a, \mathcal{E}_4, \mathcal{E}_6, \mathcal{E}_8].$$ But if $M_*(\Orth^+(L))$ is free, then $\Orth^+(L) = \Orth_r(L)$ must be generated by reflections (and in particular $a = 1$). We will use a similar argument in some other cases.
\end{remark}

\subsection{Modular forms on \texorpdfstring{$U\oplus U(2)\oplus A_2$}{} }

We will compute the graded rings of modular forms for the discriminant kernel and for the full integral orthogonal group of the lattice $L = U \oplus U(2) \oplus A_2$.

The Hilbert series of dimensions for the Weil representation $\rho$ associated to this lattice is $$\sum_{k=0}^{\infty} \mathrm{dim}\, M_{k + 1}(\rho) t^k = \frac{1 + 3t^2 + t^3 + 3t^4 + t^5 + t^6 + 2t^7}{(1 - t^4)(1 - t^6)}.$$ In particular, $\mathrm{dim}\, M_k(\rho) = 0, 1, 0, 3, 1, 4, 1$ for $k = 0, 1, 2, 3, 4, 5, 6$. For any $k \ge 3$ the spaces of cusp forms satisfy $$\mathrm{dim}\, S_k(\rho) = \begin{cases} \mathrm{dim}\, M_k(\rho) - 3: & k \; \text{odd}; \\ M_k(\rho): & k \; \text{even}; \end{cases}$$ as one can see by counting the number of Eisenstein series. In particular there are unique cusp forms of weights $4$ and $5$. We denote by $m_{2, A_2}$ the additive lift of the weight $1$ modular form and we let $m_{5, A_2}$ and $m_{6, A_2}$ denote the lifts of the weight $4$ and $5$ cusp forms. 

We will use Borcherds products whose input forms' principal parts with respect to the Gram matrix $$\begin{psmallmatrix} 0 & 0 & 0 & 0 & 0 & 1 \\ 0 & 0 & 0 &  0 & 2 & 0 \\ 0 & 0 & 2 & -1 & 0 & 0 \\ 0 & 0 & -1 & 2 & 0 & 0 \\ 0 & 2 & 0 & 0 & 0 & 0\\ 1 & 0 & 0 & 0 & 0 & 0\end{psmallmatrix}$$ are given in Table \ref{tab:U+U(2)+A2} below:

\begin{table}[htbp]
\centering
\caption{Some Borcherds products for $U \oplus U(2) \oplus A_2$} \label{tab:U+U(2)+A2}
\begin{tabular}{l*{3}{c}r}
\hline
Name & Weight & Principal part\\
\hline
$b_{1, A_2}$ & $4$ & $8 e_0 + q^{-1/2} e_{(0, 1/2, 0, 0, 1/2, 0)}$ \\
$b_{2, A_2}$ & $4$ & $8 e_0 + q^{-1/3} e_{(0, 1/2, 2/3, 1/3, 0, 0)} + q^{-1/3} e_{(0, 1/2, 1/3, 2/3, 0, 0)}$ \\
$b_{3, A_2}$ & $4$ & $8 e_0 + q^{-1/3} e_{(0, 0, 1/3, 2/3, 1/2, 0)} + q^{-1/3} e_{(0, 0, 2/3, 1/3, 1/2, 0)}$ \\
$b_{4, A_2}$ & $5$ & $10 e_0 + q^{-1/3} e_{(0, 0, 1/3, 2/3, 0, 0)} + q^{-1/3} e_{(0, 0, 2/3, 1/3, 0, 0)}$ \\
$\Phi_{25, A_2}$ & $25$ & $50 e_0 + q^{-1} e_0$ \\
\hline
\end{tabular}
\end{table}

As before, we often abbreviate $m_{j, A_2}$ and $b_{j, A_2}$ simply by $m_j$, $b_j$.

\begin{lemma} The forms $m_2, b_1, b_2, b_4, m_6$ are algebraically independent.
\end{lemma}
\begin{proof} The modular variety attached to $U \oplus U(2) \oplus A_1$ embeds as the divisor of $b_{4, A_2}$. In particular, the pullback $$P : M_*(\widetilde \Orth_r(L)) \longrightarrow M_*(\widetilde \Orth_r(U \oplus U(2) \oplus A_1))$$ is injective in weights at most $4$. The pullback is also injective in weight $6$ because $\widetilde \Orth_r(U \oplus U(2) \oplus A_1)$ admits no modular forms of weight $1$, and therefore (by injectivity) $\widetilde \Orth_r(L)$ admits no modular forms of weight $1$; and therefore no modular forms of weight $6$ that are multiples of $b_{4, A_2}$. It follows that $P(m_{2, A_2})$ is nonzero; that $P(b_{1, A_2})$ and $P(b_{2, A_2})$ span $\mathbb{C}[b_{1, A_1}^2, b_{2, A_1}^2]$ and that $P(m_{6, A_2})$ is linearly independent from $m_{2, A_1}^3, m_{2, A_1} b_{1, A_1}^2, m_{2, A_1} b_{2, A_1}^2$; and therefore that $P(m_{2, A_2})$, $P(b_{1, A_2})$, $P(b_{2, A_2})$ and $P(m_{6, A_2})$ are algebraically independent. Since $P(b_{4, A_2}) = 0$ we conclude that $m_2, b_1, b_2, b_4, m_6$ are algebraically independent. 
\end{proof}

\begin{remark} With a little more effort one can determine exact expressions for the pullbacks: we find $P(m_{2, A_2}) = m_{2, A_1}$, $P(b_{1, A_2}) = b_{4, A_1}$, $P(b_{2, A_2}) = b_{2, A_1}^2$, $P(b_{4, A_2}) = 0$, $P(m_{6, A_2}) = b_{3, A_1}^2$.
\end{remark}

\begin{theorem}\label{th:U+U2+A2_d} The graded ring of modular forms for the discriminant kernel of $U \oplus U(2) \oplus A_2$ is freely generated in weights $2, 4, 4, 5, 6$: $$M_*(\widetilde \Orth^+(L)) = \mathbb{C}[m_2, b_1, b_2, b_4, m_6].$$
\end{theorem}
\begin{proof} The Jacobian $J = J(m_2, b_1, b_2, b_4, m_6)$ is nonzero (by the previous lemma), has weight $25$, and vanishes on all mirrors associated to $\Phi_{25, A_2}$ by Theorem \ref{th:j}.  In particular it equals $\Phi_{25, A_2}$ up to a nonzero constant multiple. By Theorem \ref{th:j2} we find $$M_*(\widetilde \Orth_r(U \oplus U(2) \oplus A_2)) = \mathbb{C}[m_2, b_1, b_2, b_4, m_6].$$ Comparing dimensions shows that all modular forms for $\widetilde \Orth_r(U \oplus U(2) \oplus A_2)$ of weights at most $6$ are additive lifts, and are therefore modular under the full discriminant kernel; so we conclude that $\widetilde \Orth_r(U \oplus U(2) \oplus A_2) = \widetilde \Orth^+(U \oplus U(2) \oplus A_2).$
\end{proof}

\begin{remark} From the structure theorem it follows that there is a linear relation among the weight four products. It also follows that the additive lift $m_5$ of weight $5$ equals the Borcherds product $b_4$ (up to a multiple).
\end{remark}

\begin{remark} The discriminant form $L'/L$ contains three isotropic vectors which we label $0, v_1, v_2$. These yield three distinct vector-valued Eisenstein series $E_0, E_{v_1}, E_{v_2}$ of weight three, which can be lifted to orthogonal Eisenstein series $e_0, e_{v_1}, e_{v_2}$ respectively. By computing the first Fourier coefficients one can show that (appropriately ordered and normalized) the weight four Borcherds products are $$b_1 = e_{v_1} - e_{v_2}, \; \; b_2 = e_{v_1}, \; \; b_3 = e_{v_2},$$ and moreover that $m_2^2$ is a constant multiple of $5 e_0 + e_{v_1} + e_{v_2}$.
\end{remark}

The action of $\Orth^+(L)$ on the Borcherds products can be computed using their descriptions in terms of Eisenstein series. In particular, $b_2 + b_3$ and $b_2^2 + b_3^2$ are modular with trivial character under $\Orth^+(L)$. Moreover, the forms $m_2$, $b_4$ and $m_6$ are additive lifts of the unique modular or cusp form of the appropriate weight with rational coefficients, and are therefore modular under the larger $\Orth^+(L)$ with a character of order at most two. Thus $b_4^2$ is modular under $\Orth^+(L)$ without character. Besides, the $e_0$ components of the inputs of $m_2$ and $m_6$ are nonzero. It follows that $m_2$ and $m_6$ are also modular on $\Orth^+(L)$ without character.

\begin{theorem} The graded ring of modular forms for the integral orthogonal group of $U \oplus U(2) \oplus A_2$ is freely generated in weights $2, 4, 6, 8, 10$: $$M_*(\Orth^+(L)) = \mathbb{C}[m_2, b_2 + b_3, m_6, b_2^2 + b_3^2, b_4^2].$$ The Jacobian of the generators is a nonzero constant multiple of $b_1 b_4 \Phi_{25}$.
\end{theorem}
\begin{proof} It is clear from Theorem \ref{th:U+U2+A2_d} that these generators are also algebraically independent. Their Jacobian has weight $30$. Since $J_0 := b_1 b_4 \Phi_{25}$ also has weight $30$, and its divisor consists of a simple zero on every mirror of a reflection in $\Orth^+(L)$, we conclude from Theorems \ref{th:j} and \ref{th:j2} that the Jacobian equals $J_0$, that $\Orth^+(L)$ is generated by reflections corresponding to the divisor of $J_0$, and that $M_*(\Orth^+(L))$ has the claimed structure.
\end{proof}

\begin{remark} The above generators of weight greater than $2$ can be replaced by Eisenstein series: $$M_*(\Orth^+(L)) = \mathbb{C}[m_2, \mathcal{E}_4, \mathcal{E}_6, \mathcal{E}_8, \mathcal{E}_{10}].$$
\end{remark}

\subsection{Modular forms on \texorpdfstring{$U\oplus U(2)\oplus A_3$}{} }

The dimensions of modular forms for the Weil representation $\rho$ attached to $L := U \oplus U(2) \oplus A_3$ have the Hilbert series $$\sum_{k=0}^{\infty} \mathrm{dim}\, M_{k + 1/2}(\rho) t^k = \frac{1 + 4t^2 + t^3 + 4t^4 + t^5 + 3t^6 + 2t^7}{(1 - t^4)(1 - t^6)}.$$ In particular there is a unique normalized modular form of weight $1/2$. We label its image under the additive lift $m_{2, A_3}$. For any $k \ge 2$, counting Eisenstein series yields $$\mathrm{dim}\, S_{k + 1/2}(\rho) = \begin{cases} \mathrm{dim}\, M_{k + 1/2}(\rho) - 4: & k \; \text{even}; \\ \mathrm{dim}\, M_{k + 1/2}(\rho): & k \; \text{odd}; \end{cases}$$ and therefore there are unique (up to normalization) cusp forms of weights $7/2$ and $9/2$. We label their images under the additive lift $m_{5, A_3}$ and $m_{6, A_3}$ respectively. 

We will also use the following Borcherds products. The principal parts are given with respect to the Gram matrix $\begin{psmallmatrix} 0 & 0 & 0 & 0 & 0 & 0 & 1 \\
0 & 0 & 0 & 0 & 0 & 2 & 0 \\
0 & 0 & 2 & -1 & 0 & 0 & 0 \\
0 & 0 & -1 & 2 & -1 & 0 & 0 \\
0 & 0 & 0 & -1 & 2 & 0 & 0 \\
0 & 2 & 0 & 0 & 0 & 0 & 0 \\
1 & 0 & 0 & 0 & 0 & 0 & 0
\end{psmallmatrix}$.

\begin{table}[htbp]
\centering
\caption{Some Borcherds products for $U \oplus U(2) \oplus A_3$}
\begin{tabular}{l*{3}{c}r}
\hline
Name & Weight & Principal part\\
\hline
$b_{1, A_3}$ & $4$ & $8 e_0 + q^{-1/2} e_{(0, 0, 1/2, 0, 1/2, 1/2, 0)}$ \\
$b_{2, A_3}$ & $4$ & $8 e_0 + q^{-1/2} e_{(0, 1/2, 1/2, 0, 1/2, 0, 0)}$ \\
$b_{3, A_3}$ & $4$ & $8 e_0 +  q^{-1/2} e_{(0, 1/2, 0, 0, 0, 1/2, 0)}$ \\
$b_{4, A_3}$ & $4$ & $8 e_0 + q^{-3/8} e_{(0, 0, 3/4, 1/2, 1/4, 0, 0)} + q^{-3/8} e_{(0, 0, 1/4, 1/2, 3/4, 0, 0)}$ \\
$b_{5, A_3}$ & $4$ & $8 e_0 + q^{-3/8} e_{(0, 0, 3/4, 1/2, 1/4, 1/2, 0)} + q^{-3/8} e_{(0, 0, 1/4, 1/2, 3/4, 1/2, 0)}$ \\
$b_{6, A_3}$ & $4$ & $8 e_0 + q^{-3/8} e_{(0, 1/2, 3/4, 1/2, 1/4, 0, 0)} + q^{-3/8} e_{(0, 1/2, 1/4, 1/2, 3/4, 0, 0)}$ \\
$b_{7, A_3}$ & $5$ & $10 e_0 + q^{-1/2} e_{(0, 0, 1/2, 0, 1/2, 0, 0)}$ \\
$\Phi_{30, A_3}$ & $30$ & $60 e_0 + q^{-1} e_0$ \\
\hline
\end{tabular}
\end{table}

There is a natural embedding $A_2 \rightarrow A_3$ given by $x \mapsto (x, 0)$ if we view $A_2$ and $A_3$ as $\mathbb{Z}^2, \mathbb{Z}^3$ with Gram matrices $\begin{psmallmatrix} 2 & -1 \\ -1 & 2 \end{psmallmatrix}$ and $\begin{psmallmatrix} 2 & -1 & 0 \\ -1 & 2 & -1 \\ 0 & -1 & 2 \end{psmallmatrix}$. This induces an embedding of the modular variety associated to $U \oplus U(2) \oplus A_2$ into that of $L$ whose image is exactly the divisor of $b_{4, A_3}$. The pullbacks of the Borcherds products $b_{j, A_3}$ along this embedding are
$$P(b_{1, A_3}) = P(b_{5, A_3}) = b_{3, A_2}, \; P(b_{2, A_3}) = P(b_{6, A_3}) = b_{2, A_2},$$ $$P(b_{3, A_3}) = b_{1, A_2}, \; P(b_{4, A_3}) = 0, \; P(b_{7, A_3}) = b_{4, A_2}.$$ Similarly the pullbacks of the additive lifts are $$P(m_{2, A_3}) = m_{2, A_2},  \; P(m_{5, A_3}) = m_{5, A_2}, \; P(m_{6, A_3}) = m_{6, A_2}.$$ Using this we can prove:

\begin{theorem}\label{th:U+U2+A3_d} The graded ring of modular forms for the discriminant kernel on $U \oplus U(2) \oplus A_3$ is freely generated in weights $2, 4, 4, 4, 5, 6$: $$M_*(\widetilde \Orth^+(L)) = \mathbb{C}[m_2, b_1, b_2, b_4, b_7, m_6].$$ The Jacobian of the generators equals the Borcherds product $\Phi_{30}$ up to a nonzero constant multiple.
\end{theorem}
\begin{proof} All of the products $b_j$ are modular without character on the subgroup $\widetilde \Orth_r(L)$ generated by reflections whose mirrors lie in the divisor of $\Phi_{30, A_3}$. Since the images of $m_2, b_1, b_2, b_7, m_6$ under the pullback to $U \oplus U(2) \oplus A_2$ are generators and $b_4$ vanishes with a simple zero there, these forms are algebraically independent. By Theorem \ref{th:j} their Jacobian is $J = \Phi_{30, A_3}$, and by Theorem \ref{th:j2} $$M_*(\widetilde \Orth_r(L)) = \mathbb{C}[m_2, b_1, b_2, b_4, b_7, m_6].$$ Comparing dimensions with modular forms for the Weil representation shows that all of these generators are additive lifts, so they are modular without character under the full discriminant kernel $\widetilde \Orth^+(L)$. As before, we conclude that $\widetilde \Orth_r(L)$ is actually the discriminant kernel.
\end{proof}

\begin{remark} It follows that the weight four products $b_j$, $1 \le j \le 6$ span a three-dimensional space. By computing Fourier expansions one can see that (appropriately normalized) these products satisfy the relations $$b_1 = b_4 - b_5, \; b_2 = b_5 - b_6, \; b_3 = b_6 - b_4.$$ Moreover, if $v_1, v_2, v_3$ denote the isotropic vectors $$v_1 = (0,1/2,1/2,0,1/2,1/2,0), \; v_2 = (0,1/2, 0,0,0,0,0), \; v_3 = (0,0,0,0,0,1/2,0) \in L'/L$$ then a short computation shows that all of these products are Eisenstein series: $$b_1 = e_{v_1} - e_{v_2}, \; b_2 = e_{v_2} - e_{v_3}, \; b_3 = e_{v_3} - e_{v_1}, \; b_4 = e_{v_1}, \; b_5 = e_{v_2}, \; b_6 = e_{v_3},$$ and moreover the square of the weight two lift $m_2^2$ equals $5 e_0 + e_{v_1} + e_{v_2} + e_{v_3}$ up to a constant multiple.
\end{remark}

\begin{theorem}\label{th:U+U2+A3_f} The graded ring of modular forms for the integral orthogonal group on $U \oplus U(2) \oplus A_3$ is freely generated in weights $2, 4, 6, 8, 10, 12$: $$M_*(\Orth^+(L)) = \mathbb{C}[m_2, b_4 + b_5 + b_6, m_6, b_4^2 + b_5^2 + b_6^2, b_7^2, b_4 b_5 b_6].$$ The Jacobian of the generators is $b_1 b_2 b_3 b_7 \Phi_{30}$.
\end{theorem}
\begin{proof} Theorem \ref{th:U+U2+A3_d} shows that the claimed generators are algebraically independent. Using the action of $\Orth^+(L)$ on Eisenstein series and on the input form into $m_6$ and $b_7 = m_5$ under the additive lift, one can see that these generators are modular under $\Orth^+(L)$ without character. Their Jacobian $J$ has weight $47$, which equals the weight of the product $J_0 = b_1 b_2 b_3 b_7 \Phi_{30}$ which has a simple zero on all mirrors of reflections in $\Orth^+(L)$. As in the previous sections, Theorems \ref{th:j} and \ref{th:j2} imply that $M_*(\Orth^+(L))$ is freely generated by the forms in the claim.
\end{proof}

\begin{remark} All of the generators other than $m_2$ can be replaced by the standard Eisenstein series: $$M_*(\Orth^+(L)) = \mathbb{C}[m_2, \mathcal{E}_4, \mathcal{E}_6, \mathcal{E}_8, \mathcal{E}_{10}, \mathcal{E}_{12}].$$ This can be proved by computing the expressions of $\mathcal{E}_k$ in terms of the generators in Theorem \ref{th:U+U2+A3_f}.
\end{remark}

\subsection{Modular forms on \texorpdfstring{$U\oplus U(2)\oplus D_4$}{}}

The dimensions of spaces of modular forms for the Weil representation attached to $L = U \oplus U(2) \oplus D_4$ are given by the formula $$\sum_{k = 0}^{\infty} \mathrm{dim}\, M_k(\rho) t^k = \frac{1 + 5t^2 + 5t^4 + 5t^6}{(1 - t^4)(1 - t^6)}.$$ In particular there is a unique Weil invariant up to scalar multiple. We label its image under the additive lift $m_{2, D_4}$.

Moreover we denote by $m_{6, D_4}$ the additive lift of the modular form $f(\tau) \in M_4(\rho)$ whose Fourier expansion with respect to the Gram matrix $\begin{psmallmatrix} 0 & 0 & 0 & 0 & 0 & 0 & 0 & 1 \\
0 & 0 & 0 & 0 & 0 & 0 & 2 & 0 \\
0 & 0 & 2 & -1 & 0 & 0 & 0 & 0 \\
0 & 0 & -1 & 2 & -1 & -1 & 0 & 0 \\
0 & 0 & 0 & -1 & 2 & 0 & 0 & 0 \\
0 & 0 & 0 & -1 & 0 & 2 & 0 & 0 \\
0 & 2 & 0 & 0 & 0 & 0 & 0 & 0 \\
1 & 0 & 0 & 0 & 0 & 0 & 0 & 0
\end{psmallmatrix}$ begins 
\begin{align*} 
f(\tau) &= (1 - 16q + ...) (e_{(0, 1/2, 0, 0, 1/2, 1/2, 1/2, 0)} + e_{(0, 1/2, 1/2, 0, 0, 1/2, 1/2, 0)}) \\ &+ (128q +  ...) (e_{(0, 0, 0, 0, 0, 0, 1/2, 0)} + e_{(0, 1/2, 0, 0, 0, 0, 0, 0)} + e_{(0, 1/2, 1/2, 0, 1/2, 0, 1/2, 0)} - e_0) \\ &+ (16q^{1/2} + ...) (e_{(0, 0, 1/2, 0, 1/2, 0, 0, 0)} - e_{(0, 0, 1/2, 0, 1/2, 0, 1/2, 0)} \\
& \quad - e_{(0, 1/2, 0, 0, 0, 0, 1/2, 0)} - e_{(0, 1/2, 1/2, 0, 1/2, 0, 0, 0)}).
\end{align*}

We will also use the Borcherds products in Table \ref{tab:U+U(2)+D4}. Their principal parts are also given with respect to the Gram matrix above.

\begin{table}[htbp]
\centering
\caption{Some Borcherds products for $U \oplus U(2) \oplus D_4$} \label{tab:U+U(2)+D4}
\begin{tabular}{l*{3}{c}r}
\hline
Name & Weight & Principal part\\
\hline
$b_{1, D_4}$ & $4$ & $8 e_0 + q^{-1/2} e_{(0, 0, 0, 0, 1/2, 1/2, 0, 0)}$ \\
$b_{2, D_4}$ & $4$ & $8 e_0 + q^{-1/2} e_{(0, 0, 0, 0, 1/2, 1/2, 1/2, 0)}$ \\
$b_{3, D_4}$ & $4$ & $8 e_0 + q^{-1/2} e_{(0, 0, 1/2, 0, 0, 1/2, 0, 0)}$ \\
$b_{4, D_4}$ & $4$ & $8 e_0 + q^{-1/2} e_{(0, 0, 1/2, 0, 0, 1/2, 1/2, 0)}$ \\
$b_{5, D_4}$ & $4$ & $8 e_0 + q^{-1/2} e_{(0, 0, 1/2, 0, 1/2, 0, 0, 0)}$ \\
$b_{6, D_4}$ & $4$ & $8 e_0 + q^{-1/2} e_{(0, 0, 1/2, 0, 1/2, 0, 1/2, 0)}$ \\
$b_{7, D_4}$ & $4$ & $8 e_0 + q^{-1/2} e_{(0, 1/2, 0, 0, 0, 0, 1/2, 0)}$ \\
$b_{8, D_4}$ & $4$ & $8 e_0 + q^{-1/2} e_{(0, 1/2, 0, 0, 1/2, 1/2, 0, 0)}$ \\
$b_{9, D_4}$ & $4$ & $8 e_0 + q^{-1/2} e_{(0, 1/2, 1/2, 0, 0, 1/2, 0)}$ \\
$b_{10, D_4}$ & $4$ & $8 e_0 + q^{-1/2} e_{(0, 1/2, 1/2, 0, 1/2, 0, 0, 0)}$ \\
$\Phi_{40, D_4}$ & $40$ & $80 e_0 + q^{-1} e_0$ \\
\hline
\end{tabular}
\end{table}

The root lattice $A_3$ naturally embeds in $D_4$ by $x \mapsto (x, 0)$ with respect to the Gram matrices $\begin{psmallmatrix} 2 & -1 & 0 \\ -1 & 2 & -1 \\ 0 & -1 & 2 \end{psmallmatrix}$ and $\begin{psmallmatrix} 2 & -1 & 0 & 0 \\ -1 & 2 & -1 & -1 \\ 0 & -1 & 2 & 0 \\ 0 & -1 & 0 & 2 \end{psmallmatrix}$. Under this map the modular variety attached to $U \oplus U(2) \oplus A_3$ embeds as the divisor of $b_5$. The pullbacks of the products $b_j$ along this embedding are as follows: $$P(b_{1, D_4}) = P(b_{3, D_4}) = b_{4, A_3}, \quad P(b_{2, D_4}) = P(b_{4, D_4}) = b_{5, A_3}, \quad P(b_{5, D_4}) = 0,$$ $$P(b_{6, D_4}) = b_{1, A_3}, \quad P(b_{7, D_4}) = b_{3, A_3}, \quad P(b_{10, D_4}) = b_{2, A_3}, \quad P(b_{8, D_4}) = P(b_{9, D_4}) = b_{6, A_3}.$$ The additive lift of weight two has weight less than $4$, so by the Koecher principle its pullback is nonzero and therefore a multiple of $m_{2, A_3}$. One can compute that the theta-contraction of the form $f(\tau) \in M_4(\rho)$ is a nonzero \emph{cusp form} of weight $9/2$, so the pullback of $m_{6, D_4}$ is a nonzero cuspidal lift and therefore equals $m_{6, A_3}$ up to a nonzero multiple. 

Finally, let $\psi_{-2}(\tau) \in M^!_{-2}(\rho)$ be the input form into the product $b_{5, D_4}$. By considering its image under the Serre derivative we obtain an input form into Borcherds' singular additive theta lift $\vartheta \psi_{-2}$ whose image is a meromorphic form $h_2$ of weight $2$ with only a double pole along the modular variety $U \oplus U(2) \oplus A_3$. The leading term in the Taylor expansion of $b_{5, D_4}$ about $U \oplus U(2) \oplus A_3$ is a nonzero modular form of weight five and therefore equals $m_{5, A_3}$ (up to a nonzero multiple). The leading term in the Laurent expansion of $h_2$ about $U \oplus U(2) \oplus A_3$ is a nonzero constant. 
It follows that $h_{10} := b_{5, D_4}^2 h_2$ is a holomorphic modular form of weight $10$ whose pullback is $P(h_{10}) = m_{5, A_3}^2$.

\begin{theorem}\label{th:U+U2+D4_d} The graded ring of modular forms for the discriminant kernel of $U \oplus U(2) \oplus D_4$ is freely generated in weights $2, 4, 4, 4, 4, 6, 10$: $$M_*(\widetilde \Orth^+(L)) = \mathbb{C}[m_2, b_1, b_2, b_5, b_7, m_6, h_{10}].$$ The Jacobian of the generators equals the Borcherds product $\Phi_{40}$ up to a constant multiple.
\end{theorem}
\begin{proof} From the results on $U \oplus U(2) \oplus A_3$ we see that the pullbacks of $m_2, b_1, b_2, b_7, m_6, h_{10}$ to $U \oplus U(2) \oplus A_3$, and therefore the forms themselves, are algebraically independent. Since $b_5$ vanishes with a simple zero on $U \oplus U(2) \oplus A_3$ it follows that the forms above are algebraically independent. Their Jacobian $J$ has weight $40$ and vanishes on the divisor of $\Phi_{40}$ by Theorem \ref{th:j}, so by the Koecher principle $J / \Phi_{40}$ is constant. Letting $\widetilde \Orth_r$ be the group generated by reflections associated to $\Phi_{40}$, the Jacobian criterion implies $$M_*(\widetilde \Orth_r(L)) = \mathbb{C}[m_2, b_1, b_2, b_5, b_7, m_6, h_{10}].$$

Since $M_2(\rho)$ is $5$-dimensional, the injectivity of the additive theta lift implies that all of $m_2^2$, $b_1$, $b_2$, $b_5$, $b_7$ are additive theta lifts, and in particular the Borcherds products have trivial character on the discriminant kernel. By construction of $h_{10}$ this also implies that $h_{10}$ has trivial character on the discriminant kernel. We conclude that $\widetilde \Orth^+(L)$ is generated by reflections and that \[ M_*(\widetilde \Orth^+(L)) = \mathbb{C}[m_2, b_1, b_2, b_5, b_7, m_6, h_{10}]. \qedhere \]
\end{proof}

\begin{remark}\label{rk:U+U2+D4_f} The discriminant form $L'/L$ contains six isotropic vectors. With respect to the Gram matrix fixed above, these are $0$ and $$v_1 = (0,1/2, 1/2, 0, 1/2, 0, 1/2,0), \; v_2 = (0, 0, 0, 0, 0, 1/2, 0), \; v_3 = (0, 1/2, 0, 0, 0, 0, 0, 0, 0),$$ $$v_4 = (0, 1/2, 0, 0, 1/2, 1/2, 1/2, 0), \; v_5 = (0, 1/2, 1/2, 0, 0, 1/2, 1/2, 0).$$ The associated weight two vector-valued Eisenstein series $E_0 + E_{v_i}$, $i = 1,...,5$ are holomorphic and span $M_2(\rho)$. Denote by $e_i$ the theta lift of $E_0 + E_{v_i}$. The ten weight four Borcherds products are precisely the ten differences of the orthogonal Eisenstein series $e_i$; using the labels above (up to an ambiguous $\pm 1$ factor), $$b_1 = e_1 - e_5, \; b_2 = e_3 - e_4, \; b_3 = e_1 - e_4, \; b_4 = e_3 - e_5, \; b_5 = e_4 - e_5,$$ $$b_6 = e_1 - e_3, \; b_7 = e_2 - e_3, \; b_8 = e_2 - e_4, \; b_9 = e_2 - e_5, \; b_{10} = e_1 - e_2.$$ The action of $\Orth^+(L)$ permutes the forms $e_i$ transitively. The square $m_2^2$ is invariant under the action of $\Orth^+(L)$ and therefore must be (possibly after rescaling) the sum of all $e_i$: $$m_2^2 = e_1 + e_2 + e_3 + e_4 + e_5.$$
\end{remark}

\begin{theorem}\label{th:U+U2+D4_f} The graded ring of modular forms for the full integral orthogonal group of $U \oplus U(2) \oplus D_4$ is freely generated in weights $2, 6, 8, 10, 12, 16, 20:$ $$M_*(\Orth^+(L)) = \mathbb{C}[m_2, \mathcal{E}_6, p_2, \mathcal{E}_{10}, p_3, p_4, p_5],$$ where $$p_k = e_1^k + e_2^k + e_3^k + e_4^k + e_5^k \; \text{for} \; k \in \{2,3,4,5\}.$$ The Jacobian of the generators is a constant multiple of $\Phi_{40} \prod_{i=1}^{10} b_i$.
\end{theorem}
\begin{proof} In \cite{Sch06} Scheithauer constructed a lifting from scalar modular forms on congruence subgroups to vector-valued modular forms. Since $L$ has squarefree level $2$, Scheithauer's lift sends modular forms in $M_k(\Gamma_0(2))$ to modular forms in $M_k(\rho)$ which are invariant under all automorphisms of $L'/L$. The unique Weil invariant can be constructed as Scheithauer's lift of the constant $1$, which shows that $m_2$ is modular on $\Orth^+(L)$ without character. The modularity of the other generators in the claim under $\Orth^+(L)$ follows from Remark \ref{rk:U+U2+D4_f}.

The algebraic independence of these forms can be shown by expressing them in terms of the generators of Theorem \ref{th:U+U2+D4_d} (and only $\mathcal{E}_{10}$ requires significant computation). Therefore the Jacobian $J$ is nonzero of weight $80$. Since the product $J_0 = \Phi_{40} \prod_{i=1}^{10} b_i$ whose divisor consists of a simple zero on all mirrors of reflections in $\Orth^+(L)$ also has weight $80$, it follows from Theorem \ref{th:j} that $J = J_0$ up to a constant multiple. The claim follows from Theorem \ref{th:j2}.
\end{proof}

\begin{remark} All of the generators of weight greater than two can be replaced by the standard Eisenstein series: $$M_*(\Orth^+(L)) = \mathbb{C}[m_2, \mathcal{E}_6, \mathcal{E}_8, \mathcal{E}_{10}, \mathcal{E}_{12}, \mathcal{E}_{16}, \mathcal{E}_{20}].$$ This can be proved computationally, by expressing $\mathcal{E}_k$ in terms of the generators in Theorem \ref{th:U+U2+D4_f}.
\end{remark}

\section{The \texorpdfstring{$2U(2)\oplus R$}{} tower}\label{sec:4}

In this section we will compute the algebras of modular forms associated to the discriminant kernels of the tower of lattices $$U(2) \oplus S_8 \subseteq 2U(2) \oplus A_2 \subseteq 2U(2) \oplus A_3 \subseteq 2U(2) \oplus D_4.$$ Here $S_8$ is the signature $(2, 1)$ lattice $\mathbb{Z}^3$ with Gram matrix $\begin{psmallmatrix} -2 & 1 & 1 \\ 1 & 2 & 1 \\ 1 & 1 & 2 \end{psmallmatrix}$ and genus symbol $8_3^{+1}$ (cf. the appendix of the extended version of \cite{BEF}, where a different basis is used). Note that $S_8$ is isotropic but is not split by any $U(N)$; nevertheless it fits conveniently into the $2U(2) \oplus R$ tower, as it embeds into $2U(2) \oplus A_2$.

\subsection{Modular forms for \texorpdfstring{$U(2)\oplus S_8$}{}}

We will show that the algebra of modular forms for the discriminant kernel of $U(2) \oplus S_8$, i.e. the lattice $\mathbb{Z}^5$ with Gram matrix $\begin{psmallmatrix} 0 & 0 & 0 & 0 & 2 \\ 0 & -2 & 1 & 1 & 0 \\ 0 & 1 & 2 & 1 & 0 \\ 0 & 1 & 1 & 2 & 0 \\ 2 & 0 & 0 & 0 & 0 \end{psmallmatrix},$ is freely generated by modular forms of weights $2, 2, 2, 3$. First we compute some Borcherds products:

\begin{table}[htbp]
\centering
\caption{Some Borcherds products for $U(2) \oplus S_8$}
\begin{tabular}{l*{3}{c}r}
\hline
Name & Weight & Principal part\\
\hline
$b_{1, S_8}$ & $2$ & $4 e_0 + q^{-1/4} e_{(0,1/4,1/4,1/4,1/2)} + q^{-1/4} e_{(0, 3/4, 3/4, 3/4, 1/2)}$ \\
$b_{2, S_8}$ & $2$ & $4 e_0 +  q^{-1/4} e_{(1/2,1/4,1/4,1/4,0)} + q^{-1/4} e_{(1/2, 3/4, 3/4, 3/4, 0)}$ \\
$b_{3, S_8}$ & $2$ & $4 e_0 + q^{-1/2} e_{(1/2, 0, 0, 0, 1/2)}$ \\
$b_{4, S_8}$ & $2$ & $4 e_0 + q^{-1/2} e_{(1/2, 1/2, 1/2, 1/2, 1/2)}$ \\
$b_{5, S_8}$ & $2$ & $4 e_0 + q^{-5/16} e_{(0,7/8,3/8,3/8,0)} + q^{-5/16} e_{(0, 1/8, 5/8, 5/8, 0)}$ \\
$b_{6, S_8}$ & $2$ & $4 e_0 + q^{-5/16} e_{(0,7/8,3/8,3/8,1/2)} + q^{-5/16} e_{(0, 1/8, 5/8, 5/8, 1/2)}$ \\
$b_{7, S_8}$ & $2$ & $4 e_0 + q^{-5/16} e_{(1/2,7/8,3/8,3/8,0)} + q^{-5/16} e_{(1/2, 1/8, 5/8, 5/8, 0)}$ \\
$b_{8, S_8}$ & $2$ & $4 e_0 + q^{-5/16} e_{(1/2,5/8,1/8,1/8,1/2)} + q^{-5/16} e_{(1/2, 3/8, 7/8, 7/8, 1/2)}$ \\
$\psi_{S_8}$ & $3$ & $6 e_0 + q^{-1/4} e_{(0, 1/4, 1/4, 1/4, 0)} + q^{-1/4} e_{(0, 3/4, 3/4, 3/4, 0)}$ \\
$\Phi_{12, S_8}$ & $12$ & $24e_0 + q^{-1} e_0$ \\
\hline
\end{tabular}
\end{table}

We omit $S_8$ from the index as long as there is no risk of confusion.

\begin{theorem} The eight Borcherds products $b_1,...,b_8$ of weight two span a three-dimensional space and they satisfy the three-term relations $$b_1 = b_2 + b_3 = b_4 - b_2 = b_5 - b_6 = b_7 - b_8 \; \text{and} \; b_2 = b_5 - b_7.$$ All of the products $b_1,...,b_8$ and $\psi$ can be realized as additive lifts and are therefore modular without character on the discriminant kernel of $U(2) \oplus S_8$.  Any three linearly independent weight two products (e.g. $b_1, b_2, b_5$) are algebraically independent and together with $\psi$ they freely generate the algebra of modular forms: $$M_*(\widetilde \Orth^+(L)) = \mathbb{C}[b_1, b_2, b_5, \psi].$$ The Jacobian of the generators is a nonzero multiple of the Borcherds product $\Phi_{12}$.
\end{theorem}
\begin{proof} All of the weight two products transform without character on $\widetilde \Orth_r(L)$, the group generated by reflections associated to $\Phi_{12}$. (This can be read off of their divisors.) By computing the first Fourier coefficients we find that the Jacobian $J = J(b_1, b_2, b_5, \psi)$ is not identically zero, and it also has weight $12$. As in the previous sections, Theorems \ref{th:j} and \ref{th:j2} imply that $J$ equals $\Phi_{12}$ up to a constant multiple, and that $b_1, b_2, b_5, \psi$ freely generate the algebra of modular forms for $\widetilde \Orth_r(L)$. Computing dimensions of the spaces of additive lifts (taking into account that the additive lift is not injective) shows that all of the generators arise as additive lifts, from which we conclude that the generators are modular under the discriminant kernel $\widetilde \Orth^+(L)$ and finally that $\widetilde \Orth^+(L)$ is generated by reflections.
\end{proof}

\begin{remark} The space of modular forms of weight $3/2$ for the Weil representation attached to $L$ is spanned by Eisenstein series. There are six isotropic vectors in the discriminant form $L'/L$. In addition to $0$, with respect to the Gram matrix fixed above, they are represented by $$v_1 = (0, 1/2, 1/2, 1/2, 0), \; v_2 = (0, 0, 0, 0, 1/2), \; v_3 = (0, 1/2, 1/2, 1/2, 1/2),$$ $$v_4 = (1/2, 0, 0, 0, 0), \; v_5 = (1/2, 1/2, 1/2, 1/2, 0).$$ Attached to any of these cosets one can associate a mock Eisenstein series $E_0, E_{v_i}$. Computing their shadows shows that the space of holomorphic forms of weight $3/2$ is spanned by the linear combinations $$E_0 - E_{v_1}, \; E_0 - E_{v_2} - E_{v_5}, \; E_0 - E_{v_3} - E_{v_4}, \; E_0 - E_{v_3} - E_{v_5}.$$ The additive theta lift to orthogonal modular forms of weight two has a kernel: $$\mathrm{ker}\Big( M_{3/2}(\rho) \longrightarrow \mathrm{Maass}_2(\widetilde \Orth^+(L)) \Big) = \mathrm{span}\Big( 3 E_0 - E_{v_1} - E_{v_2} - E_{v_3} - E_{v_4} - E_{v_5} \Big).$$ By computing Fourier expansions one can show that (up to a sign ambiguity) the weight two Borcherds products are the following explicit additive lifts:
\begin{align*}
&b_1 = \mathrm{Lift}(E_{v_4} - E_{v_5}),& & b_2 = \mathrm{Lift}(E_{v_2} - E_{v_3}),& \\
&b_3 = \mathrm{Lift}(E_{v_3} + E_{v_4} - E_{v_2} - E_{v_5}),&   & b_4 = \mathrm{Lift}(E_{v_2} + E_{v_4} - E_{v_3} - E_{v_5}),&\\
&b_5 = \mathrm{Lift}(E_{v_1} - E_{v_3} - E_{v_5}),& & b_6 = \mathrm{Lift}(E_{v_1} - E_{v_3} - E_{v_4}), & \\
&b_7 = \mathrm{Lift}(E_{v_1} - E_{v_2} - E_{v_5}),& & b_8 = \mathrm{Lift}(E_{v_1} - E_{v_2} - E_{v_4}).&
\end{align*}
(Note that the input form in each of the above lifts is holomorphic.)

The action of $\Orth^+(L)$ maps Borcherds products to other products with rational quadratic divisors of the same norm. In particular it permutes the products $b_5,b_6,b_7,b_8$. On the other hand, any $g \in \Orth^+(L)$ sends the lift of the Eisenstein series $E_v$ to the lift of $E_{g \cdot v}$. Considering the expressions for $b_5,...,b_8$ as additive lifts shows that $\Orth^+(L)$ permutes $b_5,...,b_8$ without multiplication by any roots of unity and that the coset $v_1 \in L'/L$ is invariant under $\Orth^+(L)$. (This can also be shown directly.) In particular, $$\mathrm{Lift}(-3 E_0 + 3E_{v_1}) = \mathrm{Lift}(2E_{v_1} - E_{v_2} - E_{v_3} - E_{v_4} - E_{v_5}) = b_5 + b_8 = b_6 + b_7$$ is modular under the entire group $\Orth^+(L)$.
\end{remark}

\begin{theorem} The graded ring of modular forms for the full integral orthogonal group of $U(2) \oplus S_8$ is freely generated in weights $2, 4, 6, 8$: $$M_*(\Orth^+(L)) = \mathbb{C}[p_1, p_2, \psi^2, p_4], \; \text{where} \; p_k = b_5^k + b_6^k + b_7^k + b_8^k.$$ The Jacobian of the generators is a constant multiple of $b_1 b_2 b_3 b_4 \psi \Phi_{12}$. 
\end{theorem}
Note that we must omit the third power sum $\sum_{i=5}^8 b_i^3$ from the generators due to the relation $$\Big( \sum_{i=5}^8 b_i \Big)^3 - 6 \Big( \sum_{i=5}^8 b_i \Big) \Big( \sum_{i=5}^8 b_i^2 \Big) + 8 \sum_{i=5}^8 b_i^3 = 0.$$ (This follows from the relation $b_5 + b_8 = b_6 + b_7$.)
\begin{proof} Using $b_5 + b_8 = b_6 + b_7$ and the fact that $b_5, b_6, b_7, \psi$ are algebraically independent, it is straightforward to show that the generators in the claim are also algebraically independent. The previous remark implies that all of the power sums are modular under $\Orth^+(L)$. We obtain the transformation of $\psi^2$ under $\Orth^+(L)$ using the action of $\Orth^+(L)$ on the input form into $\psi$ in the additive lift.
\end{proof}

\begin{remark} If we define $\mathcal{E}_2$ as the lift of the holomorphic invariant Eisenstein series $E_0 - E_{v_1}$ then we can also generate the algebra of modular forms with only Eisenstein series: $$M_*(\Orth^+(L)) = \mathbb{C}[\mathcal{E}_2, \mathcal{E}_4, \mathcal{E}_6, \mathcal{E}_8].$$
\end{remark}

\subsection{Modular forms for \texorpdfstring{$2U(2)\oplus A_2$}{}}

In this section we will compute the ring of modular forms for the discriminant kernel of the lattice $L = 2U(2) \oplus A_2$ using a pullback map to $U(2) \oplus S_8$. Using the Gram matrices $\begin{psmallmatrix} -2 & 1 & 1 \\ 1 & 2 & 1 \\ 1 & 1 & 2 \end{psmallmatrix}$ for $S_8$ and $\begin{psmallmatrix} 0 & 0 & 0 & 2 \\ 0 & 2 & -1 & 0 \\ 0 & -1 & 2 & 0 \\ 2 & 0 & 0 & 0 \end{psmallmatrix}$ for $U(2) \oplus A_2$ one can check that $$(x, y, z) \in S_8 \mapsto (-x+y+z, x - y, z - x, 2x) \in U(2) \oplus A_2$$ is an isometric embedding. By acting trivially on the extra copy of $U(2)$ this extends to an embedding of the modular varieties. 

This lattice admits $15$ holomorphic Borcherds products of weight $2$. Six have principal parts of the form $$4 e_0 + q^{-1/2} e_v, \; \; \mathrm{ord}(v) = 2$$ and the remaining nine have principal parts of the form $$4 e_0 + q^{-1/3} e_v + q^{-1/3} e_{-v}, \; \; \mathrm{ord}(v) = 6.$$ Under the embedding above, the variety associated to $U(2) \oplus S_8$ embeds as the divisor of the Borcherds product whose principal part is $$4 e_0 + q^{-1/3} e_{(0, 1/2, 1/3, 2/3, 0, 0)} + q^{-1/3} e_{(0, 1/2, 2/3, 1/3, 0, 0)}.$$ In addition there is a Borcherds product $\psi$ of weight $3$ with principal part $$6 e_0 + q^{-1/3} e_{(0, 0, 1/3, 2/3, 0, 0)} + q^{-1/3} e_{(0, 0, 2/3, 1/3, 0, 0)}$$ and a product $\Phi_{15}$ of weight $15$ whose principal part is $30 e_0 + q^{-1} e_0$.

\begin{theorem}
\noindent
\begin{itemize} 
\item[(1)] The 15 Borcherds products of weight $2$ span a four-dimensional space. All of the products of weight $2$ and $3$ can also be constructed as additive lifts. 
\item[(2)] The discriminant kernel $\widetilde \Orth^+(L)$ is generated by reflections. 
\item[(3)]  If $b_1, b_2, b_3, b_4$ are any linearly independent products of weight $2$, then they are algebraically independent and together with the weight $3$ product $\psi$ they freely generate the ring of modular forms for the discriminant kernel: $$M_*(\widetilde \Orth^+(L)) = \mathbb{C}[b_1, b_2, b_3, b_4, \psi].$$ 
\end{itemize}
\end{theorem}
\begin{proof} By computing the theta-contraction of the inputs into the $14$ Borcherds products of weight $2$ that do not vanish along $U(2) \oplus S_8$, we find all $8$ of the weight two products on $U(2) \oplus S_8$; and similarly the weight three product on $2U(2) \oplus A_2$ pulls back to the weight three product on $U(2) \oplus S_8$. Any products that pull back to generators of the algebra of modular forms of $U(2) \oplus S_8$ are algebraically independent; and if we add the product with a simple zero along $U(2) \oplus S_8$ then the set remains algebraically independent. The Jacobian $J$ of this set of products of weights $2,2,2,2,3$ (which we label $b_1, b_2, b_3, b_4, \psi$) has weight $15$, matching the product $\Phi_{15}$ whose divisor consists exactly of simple zeros of mirrors of reflections in the discriminant kernel. Using Theorems \ref{th:j} and \ref{th:j2} we obtain $J = \Phi_{15}$ up to a multiple and $M_*(\widetilde \Orth_r(L)) = \mathbb{C}[b_1,b_2,b_3,b_4,\psi]$. The Weil representation attached to $L$ admits a five-dimensional space of modular forms of weight $1$ which map to a four-dimensional space of forms under the additive theta lift; and comparing dimensions shows that all of the products of weight two are additive lifts and are therefore modular under the full discriminant kernel without character. Similarly, there is a vector-valued cusp form of weight two whose additive theta lift is nonzero, and therefore equals $\psi$ (up to a multiple).
\end{proof}

\begin{remark} The weight two cusp form for $\rho_L$ whose theta lift is $\psi$ takes the form $$ \eta(\tau)^4 \sum_{\substack{v \in L'/L \\ Q(v) = -1/6 + \mathbb{Z}}} \varepsilon(v) \mathfrak{e}_v$$ with certain coefficients $\varepsilon(v) \in \{-1, 1\}$, where $\eta(\tau) = q^{1/24} \prod_{n=1}^{\infty} (1 - q^n)$.
\end{remark}

We will now consider the algebra associated to the maximal reflection group $\Orth_r(L)$, generated by all reflections in $\Orth^+(L)$. (It will turn out that $\Orth^+(L)$ is \emph{not} generated by reflections: the map that swaps the two copies of $U(2)$ is not contained in $\Orth_r(L)$.) Clearly the Eisenstein series $\mathcal{E}_4, \mathcal{E}_6$ of weights $4$ and $6$ are contained in this ring. The Weil representation attached to $L$ admits a one-dimensional space of cusp forms of weight three, spanned by the form $$f(\tau) =  \Big( \eta(\tau/3)^3 \eta(\tau)^3 + 3 \eta(\tau)^3 \eta(3 \tau)^3 \Big) \sum_{\substack{v \in L'/L \\ Q(v) = -1/6 + \mathbb{Z}}} \varepsilon(v) \mathfrak{e}_v + \eta(\tau)^3 \eta(3\tau)^3 \sum_{\substack{v \in L'/L \\ Q(v) = -1/2 + \mathbb{Z}}} \varepsilon(v) \mathfrak{e}_v$$ for some coefficients $\varepsilon(v) \in \{\pm 1\}$. This form is preserved by the action of $\Orth_r(L)$ on $\mathbb{C}[L'/L]$. (It is not preserved under swapping the two copies of $U(2)$, which instead sends $f$ to $-f$!) By taking the Serre derivative of $f$, we obtain a cusp form of weight $5$, $$\vartheta f(\tau) = \frac{1}{2\pi i} f'(\tau) - \frac{1}{4} f(\tau) E_2(\tau), \; \; E_2(\tau) = 1 - 24 \sum_{n=1}^{\infty} \sigma_1(n) q^n,$$ which is also invariant under $\Orth_r(L)$ (and which transforms under $\Orth^+(L)$ with the same quadratic character). We let $p_4 \in S_4(\Orth_r(L))$ and $p_6 \in S_6(\Orth_r(L))$ be the theta lifts of $f$ and $\vartheta f$.

\begin{theorem} The algebra of modular forms for $\Orth_r(L)$ is freely generated in weights $4, 4, 6, 6, 6$: $$M_*(\Orth_r(L)) = \mathbb{C}[\mathcal{E}_4, p_4, \mathcal{E}_6, p_6, \psi^2].$$ The Jacobian of the generators is a constant multiple of $b_1 b_2 b_3 b_4 b_5 b_6 \psi \Phi_{15}$, where $b_1,...,b_6$ are the weight two Borcherds products with principal parts of the form $4e_0 + q^{-1/2} e_v$, $\mathrm{ord}(v) = 2$.
\end{theorem}

\begin{corollary} The ring of modular forms for the full integral orthogonal group of $L$ is generated by $\psi^2$ and by the Eisenstein series of weights $4, 6, 8, 10, 12$ with a single relation in weight $20$: $$M_*(\Orth^+(L)) = \mathbb{C}[\mathcal{E}_4, \mathcal{E}_6, \mathcal{E}_8, \mathcal{E}_{10}, \mathcal{E}_{12}, \psi^2] / R_{20}.$$
\end{corollary}
\begin{proof} Since $p_4$ and $p_6$ transform under $\Orth^+(L)$ with the same quadratic character, we find $$M_*(\Orth^+(L)) = \mathbb{C}[\mathcal{E}_4, \mathcal{E}_6, \psi^2, p_4^2, p_4 p_6, p_6^2] / R_{20}$$ with the relation $p_4^2 \cdot p_6^2 = (p_4 p_6)^2$ in weight $20$. By computing the expressions for the Eisenstein series in terms of the generators for $M_*(\Orth_r(L))$, we find that $\mathcal{E}_4, \mathcal{E}_6, \psi^2, \mathcal{E}_8, \mathcal{E}_{10}, \mathcal{E}_{12}$ also generate the algebra $M_*(\Orth^+(L))$.
\end{proof}

\subsection{Modular forms for \texorpdfstring{$2U(2)\oplus A_3$}{}}

The lattice $L = 2U(2) \oplus A_3$ admits $25$ Borcherds products of weight two. 15 of them have principal parts of the form $$4 e_0 + q^{-1/2} e_v, \; \mathrm{ord}(v) = 2,$$ and the remaining $10$ have principal parts of the form $$4e_0 + q^{-3/8} e_v + q^{-3/8} e_{-v}, \; \mathrm{ord}(v) = 4.$$

There is a natural embedding $2U(2) \oplus A_2 \rightarrow 2U(2) \oplus A_3$. If we use the Gram matrices $\begin{psmallmatrix} 2 & -1 \\ -1 & 2 \end{psmallmatrix}$ and $\begin{psmallmatrix} 2 & -1 & 0 \\ - 1 & 2 & -1 \\ 0 & -1 & 2 \end{psmallmatrix}$ for the root lattices then this is induced by the map $x \mapsto (x, 0)$. With respect to these Gram matrices and this embedding, the modular variety attached to $2U(2) \oplus A_2$ embeds as the divisor of the Borcherds product with principal part $$4 e_0 + q^{-3/8} e_{(0, 0, 3/4, 1/2, 1/4, 0, 0)} + q^{-3/8} e_{(0, 0, 1/4, 1/2, 3/4, 0, 0)}.$$ There is again a unique Borcherds product $\psi$ of weight $3$, this time with principal part $$6e_0 + q^{-1/2} e_{(0, 0, 1/2, 0, 1/2, 0, 0)}.$$ Finally, there is a Borcherds product $\Phi_{18}$ of weight $18$ with divisor $36e_0 + q^{-1} e_0$.

\begin{theorem}\label{th:2U(2)+A3_d} The $25$ Borcherds products of weight $2$ span a five-dimensional space. All of the products of weight $2$ and $3$ are also additive lifts. If $b_1, b_2, b_3, b_4, b_5$ are any linearly independent products of weight $2$, then they are algebraically independent and together with the weight $3$ product $\psi$ they generate the algebra of modular forms for the discriminant kernel: $$M_*(\widetilde \Orth^+(L)) = \mathbb{C}[b_1,b_2,b_3,b_4,b_5, \psi].$$
\end{theorem}
\begin{proof} Computing the theta-contraction of the inputs into the $24$ products of weight two that do not vanish along $2U(2) \oplus A_2$ shows that all of the $15$ weight two products on that lattice arise as pullbacks. Similarly the weight three products on $2U(2) \oplus A_3$ pulls back to the weight three product on $2U(2) \oplus A_2$. By essentially the same argument as the pullback from $2U(2) \oplus A_2$ to $U(2) \oplus S_8$, we obtain the algebraic independence of any linearly independent set of weight two products and the fact that, together with the weight three product, these generate the ring of modular forms for the group $\widetilde\Orth_r$ generated by reflections whose mirrors lie in the divisor of $\Phi_{18}$ (and that their Jacobian equals $\Phi_{18}$ up to a nonzero multiple). There is a six-dimensional space of modular forms of weight $1/2$ for the Weil representation attached to $2U(2) \oplus A_3$, which lift to five linearly independent forms under the additive theta lift, and therefore all of the products of weight $2$ are additive lifts (and therefore modular without character on the discriminant kernel); similarly, the weight three product is the additive lift of the unique (up to multiple) vector-valued cusp form of weight $3/2$ for the Weil representation attached to $L$.
\end{proof}

\begin{remark} The cusp form of weight $3/2$ for the Weil representation $\rho_L$ is obtained by ``averaging out'' the form $\eta(\tau)^3$: it takes the form $$f(\tau) = \eta(\tau)^3 \sum_{\substack{v \in L'/L \\ Q(v) = -1/8 + \mathbb{Z}}} \varepsilon(v) e_v,$$ where $\varepsilon(v) \in \{\pm 1\}$. In particular $f$ and therefore $\psi_3$ transforms with a quadratic character under $\Orth^+(L)$.
\end{remark}

Finally, we determine the graded ring of modular forms for $\Orth^+(L)$.

\begin{theorem}\label{th:2U(2)+A3} The group $\Orth^+(L)$ is generated by reflections. The algebra of modular forms for $\Orth^+(L)$ is freely generated in weights $4, 6, 6, 8, 10, 12$: $$M_*(\Orth^+(L)) = \mathbb{C}[\mathcal{E}_4, \mathcal{E}_6, \mathcal{E}_8, \mathcal{E}_{10}, \mathcal{E}_{12}, \psi^2].$$ The Jacobian of the generators is a constant multiple of $\Phi_{18} \psi \cdot \prod_{i=1}^{15} b_i$, where $b_i$ are the 15 Borcherds products of weight two with principal part $4e_0 + q^{-1/2} e_v$, $\mathrm{ord}(v) = 2$.
\end{theorem}

\begin{remark}
Theorem \ref{th:2U(2)+A3} has been proved by Freitag--Salvati Manni in \cite[Corollary 4.7]{FSM07}. Moreover, the modular group in \cite[Theorem 4.6]{FSM07} is exactly the reflection subgroup $\Orth_1(L)$ generated by the discriminant kernel and the reflections associated to the divisor of $\psi$. It is easy to derive that 
$$
M_*(\Orth_1(L))=\CC[b_1,b_2,b_3,b_4,b_5,\psi^2],
$$
which recovers \cite[Theorem 4.6]{FSM07}.
\end{remark}

\subsection{Modular forms for \texorpdfstring{$2U(2)\oplus D_4$}{}}

The lattice $L = 2U(2) \oplus D_4$ admits $36$ Borcherds products of singular weight $2$, all with principal parts of the form $$4 e_0 + q^{-1/2} e_v, \; \mathrm{ord}(v) = 2.$$ Using the Gram matrices $\begin{psmallmatrix} 2 & -1 & 0 \\ -1 & 2 & -1 \\ 0 & -1 & 2 \end{psmallmatrix}$ and $\begin{psmallmatrix} 2 & -1 & 0 & 0 \\ -1 & 2 & -1 & -1 \\ 0 & -1 & 2 & 0 \\ 0 & -1 & 0 & 2 \end{psmallmatrix}$ for the root lattices $A_3$ and $D_4$, there is a natural embedding $x \mapsto (x, 0)$ of $A_3$ in $D_4$ that induces an embedding of the modular variety attached to $2U(2) \oplus A_3$ into that for $2U(2) \oplus D_4$. The image is exactly the divisor of the Borcherds product with principal part $$4e_0 + q^{-1/2} e_{(0, 0, 1/2, 0, 1/2, 0, 0, 0)}.$$ Finally, there is a Borcherds product $\Phi_{24}$ of weight $24$ with principal part $48 e_0 + q^{-1} e_0$.

The dimensions of modular forms for the Weil representation attached to $2U(2) \oplus D_4$ have generating series $$\sum_{k = 0}^{\infty} \mathrm{dim}\, M_k(\rho) t^k = \frac{7 + 21 t^2 + 21t^4 + 15t^6}{(1 - t^4)(1 - t^6)}.$$ In particular there is a seven-dimensional space of Weil invariants. 

Unlike the lower-rank cases in this tower, there are no modular forms of weight three (or indeed any odd weight) so the weight three product $\psi$ on $2U(2) \oplus A_3$ is not contained in the image of the pullback. However we can construct a preimage of $\psi^2$ as follows. Let $f(\tau) \in M_{-2}^!(\rho)$ be the input form that yields the Borcherds product $h_1$ with a simple zero on $2U(2) \oplus A_3$. Then the Serre derivative of $f$ maps under the singular additive theta lift to a meromorphic modular form $h_2$ of weight $2$ with only a double pole on $2U(2) \oplus A_3$. 
The leading term in the Taylor expansion of $h_1$ about $2U(2) \oplus A_3$ is a nonzero modular form of weight three and therefore equals $\psi$ (up to a nonzero multiple). The leading term in the Laurent expansion of $h_2$ about $2U(2) \oplus A_3$ is a nonzero constant. Thus $h := h_1^2 h_2$ is a holomorphic modular form of weight $6$ whose pullback to $2U(2) \oplus A_3$ equals $\psi^2$ up to a constant multiple.

\begin{theorem} The $36$ Borcherds products of weight $2$ span a six-dimensional space. If $b_1,...,b_6$ are any linearly independent products of weight $2$, then they are algebraically independent and together with the weight six form $h$ they generate the algebra of modular forms for the discriminant kernel: $$M_*(\widetilde \Orth^+(L)) = \mathbb{C}[b_1,b_2,b_3,b_4,b_5,b_6,h].$$ The Jacobian $J = J(b_1,b_2,b_3,b_4,b_5,b_6,h)$ is a nonzero constant multiple of the product $\Phi_{24}$.
\end{theorem}
\begin{proof} Choose any five products $b_1,...,b_5$ which restrict to a spanning set of $M_2(\widetilde \Orth^+(2U(2) \oplus A_3))$ and let $b_6$ be the product that vanishes along $2U(2) \oplus A_3$. Then the restrictions of $b_1,...,b_5$ and $h$ are algebraically independent; so $b_1,...,b_5,b_6,h$ are algebraically independent. Therefore their Jacobian $J = J(b_1,...,b_6, h)$ is nonzero of weight $24$. Applying Theorems \ref{th:j}, \ref{th:j2} as in the previous sections, we conclude that $J = \Phi_{24}$ up to a constant multiple and that $$M_*(\widetilde \Orth_r(L)) = \mathbb{C}[b_1,...,b_6, h]$$ where $\widetilde \Orth_r(L)$ is generated by reflections whose mirrors lie in the divisor of $\Phi_{24}$. The Weil invariants lift to a six-dimensional space of additive lifts of forms of weight two, so comparing dimensions shows that all of the weight two Borcherds products are also additive lifts, and are therefore modular without character under the full discriminant kernel. By construction of $h$ this implies that $h$ is also modular without character under the discriminant kernel. As before, we conclude that $\widetilde \Orth^+(L)$ is generated by its reflections.
\end{proof}

\begin{remark}
This result was first proved by Freitag--Salvati Manni \cite{FSM07} using a geometric approach. They also showed in \cite{FH00} that the 36 singular weight Borcherds products vanish at a distinguished point simultaneously.
\end{remark}

Since $D_4$ has level two, passing from $L$ to $2L'$ preserves the full orthogonal group i.e. $$\Orth^+(2U(2) \oplus D_4) = \Orth^+(2U \oplus D_4),$$ so $M_*(\Orth^+(2U(2)\oplus D_4)) \cong M_*(\Orth^+(2U \oplus D_4))$.

\section{The \texorpdfstring{$U\oplus U(3)\oplus R$}{} tower}\label{sec:5}

\subsection{The \texorpdfstring{$U\oplus U(3)\oplus A_1$}{} lattice}

Let $L = U \oplus U(3) \oplus A_1$. We will compute two free algebras of orthogonal modular forms attached to $L$. Note that the discriminant kernel of $L$ can be interpreted as a level three subgroup of $\mathrm{Sp}_4(\mathbb{Z})$, and that these algebras were found by Aoki--Ibukiyama \cite{AI}. We include our argument as it is somewhat simpler.

Since $L$ splits a unimodular plane over $\mathbb{Z}$, the additive theta lift is injective, and the dimensions of the Maass subspaces equal the dimensions of modular forms for the Weil representation: $$\sum_{k=0}^{\infty} \mathrm{dim}\, \mathrm{Maass}_{k + 1}(\widetilde \Orth^+(L)) t^k = \sum_{k=0}^{\infty} \mathrm{dim}\, M_{k + 1/2}(\rho) t^k = \frac{(1 + t + t^3)(1 + t^2 + t^3)}{(1 - t^2) (1 - t^6)}.$$ Moreover, counting Eisenstein series shows that for $k \ge 2$ $$\mathrm{dim}\, S_{k + 1/2}(\rho) = \begin{cases} \mathrm{dim}\, M_{k + 1/2}(\rho) - 3: & k \; \text{odd}; \\ \mathrm{dim}\, M_{k + 1/2}(\rho) - 2: & k \; \text{even}. \end{cases}$$ In particular, there is a lift $m_1$ of weight one, two linearly independent lifts $m_3^{(1)}, m_3^{(2)}$ of weight three, and a cuspidal lift $m_4$ of weight $4$.

With respect to the Gram matrix $\begin{psmallmatrix} 0 & 0 & 0 & 0 & 1 \\ 0 & 0 & 0 & 3 & 0 \\ 0 & 0 & 2 & 0 & 0 \\ 0 & 3 & 0 & 0 & 0 \\ 1 & 0 & 0 & 0 & 0 \end{psmallmatrix}$, this lattice admits Borcherds products with principal parts as in Table \ref{tab:U+U3+A1}.

\begin{table}[htbp]
\centering
\caption{Some Borcherds products for $U\oplus U(3) \oplus A_1$}
\label{tab:U+U3+A1}
\begin{tabular}{l*{3}{c}r}
\hline
Name & Weight & Principal part\\
\hline
$b_{1, A_1}$ & $2$ & $4 e_0 + q^{-1/4} e_{(0,0,1/2,0,0)}$ \\
$b_{2, A_1}$ & $3$ & $6 e_0 + q^{-1/4} e_{(0,0,1/2,1/3,0)} + q^{-1/4} e_{(0,0,1/2,2/3,0)}$ \\
$b_{3, A_1}$ & $3$ & $6 e_0 + q^{-1/4} e_{(0,1/3,1/2,0,0)} + q^{-1/4} e_{(0,2,/31/2,0,0)}$ \\
$\psi_{A_1}$ & $3$ & $6 e_0 + q^{-1/3} e_{(0,1/3,0,1/3,0)} + q^{-1/3} e_{(0,2/3,0,2/3,0)}$ \\
$\Phi_{14, A_1}$ & $14$ & $28e_0 + q^{-1} e_0$ \\
\hline
\end{tabular}
\end{table}

\begin{lemma} The Jacobian $J(m_1, m_3^{(1)}, m_3^{(2)}, m_4)$ is not identically zero.
\end{lemma}
\begin{proof} This can be checked by computing the Fourier expansions of $m_1, m_3^{(1)}, m_3^{(2)}, m_4$ to precision $O(q, s)^5$.
\end{proof}

\begin{theorem}\label{th:U+S8_d} The graded ring of modular forms for the discriminant kernel of $U\oplus U(3)\oplus A_1$ is a free algebra: $$M_*(\widetilde \Orth^+(L)) = \mathbb{C}[m_1, m_3^{(1)}, m_3^{(2)}, m_4].$$ The Jacobian of the generators equals $\Phi_{14}$ up to a constant multiple.
\end{theorem}
\begin{proof} Since the Jacobian $J = J(m_1, m_3^{(1)}, m_3^{(2)}, m_4)$ is nonzero of weight $14$, Theorem \ref{th:j} implies that $J / \Phi_{14}$ is a nonzero constant. From Theorem \ref{th:j2} we conclude that $M_*(\widetilde \Orth_r(L)) = \mathbb{C}[m_1, m_3^{(1)}, m_3^{(2)}, m_4]$. Since all of the generators are additive lifts, they are modular under the discriminant kernel, and therefore $\widetilde \Orth_r(L) = \widetilde \Orth^+(L)$.
\end{proof}

\begin{remark} From its divisor it is clear that $b_1$ is not a square; in particular, it is not a multiple of $m_1^2$. In fact, it has a quadratic character on $\widetilde \Orth^+(L)$. By contrast all of the products $b_2, b_3, \psi$ are additive lifts.
\end{remark}

The subgroup $\Orth_r(L)$ generated by all reflections in $\Orth^+(L)$ corresponds to the divisor of $\psi \Phi_{14}$. (Note that $\Phi_{14}$ is already a multiple of $b_1$.) From its input form one can see that $m_1$ is already modular under $\Orth_r(L)$. We can produce the unique (up to scalar) modular form $m_3$ of weight three for $\Orth_r(L)$ by writing $m_1 = \Phi_f$ as a theta lift of some $f \in M_{1/2}(\rho_L)$, and defining $m_3 := \Phi_{\vartheta f}$ as the lift of the Serre derivative $$\vartheta f(\tau) = \frac{1}{2\pi i} f'(\tau) - \frac{1}{24} f(\tau) E_2(\tau).$$

\begin{theorem} The graded ring of modular forms for the maximal reflection group $\Orth_r(L)$ is freely generated in weights $1, 3, 4, 6$: $$M_*(\Orth_r(L)) = \mathbb{C}[m_1, m_3, m_4, \psi^2].$$ The Jacobian of the generators equals $\psi \Phi_{14}$ up to a constant multiple.
\end{theorem}
\begin{proof} The form $m_4$ is modular without character on $\Orth^+(L)$ because it is the theta lift of a weight $7/2$ vector-valued cusp form $f(\tau) = \sum_{x \in L'/L} f_x(\tau) e_x$ whose component $f_x$ depends only on the norm of $x$. From Theorem \ref{th:U+S8_d} it follows immediately that the Jacobian of the generators is a nonzero multiple of $\psi \Phi_{14}$, and from Theorem \ref{th:j2} we conclude that $M_*(\Orth_r(L)) = \mathbb{C}[m_1,m_3,m_4,\psi^2].$
\end{proof}

\begin{corollary} The graded ring of modular forms for $\Orth^+(L)$ is generated in weights $2, 4, 4, 6, 6$ with a single relation in weight $8$: $$M_*(\Orth^+(L)) = \mathbb{C}[m_1^2, \mathcal{E}_4, m_4, \mathcal{E}_6, \psi^2] / R_8,$$ where $\mathcal{E}_k$ is the Eisenstein series of weight $k$.
\end{corollary}
\begin{proof} Since $m_1$ and $m_3$ transform with the same quadratic character under $\Orth^+(L)$, it follows that $$M_*(\Orth^+(L)) = \mathbb{C}[m_1^2, m_1 m_3, m_4, m_3^2, \psi^2] / R_8$$ where the relation is $(m_1^2) \cdot (m_3^2) = (m_1 m_3)^2.$ Using Fourier expansions one can show that the Eisenstein series $\mathcal{E}_4, \mathcal{E}_6$ can be taken as generators instead of $m_1m_3$ and $m_3^2$ (i.e. neither of $\mathcal{E}_4, \mathcal{E}_6$ is a multiple of $m_1$).
\end{proof}

\subsection{The \texorpdfstring{$U\oplus U(3)\oplus A_2$}{} lattice}

We will compute modular forms for the discriminant kernel of the lattice $L = U \oplus U(3) \oplus A_2$ with Gram matrix $\begin{psmallmatrix} 0 & 0 & 0 & 0 & 0 & 1 \\ 0 & 0 & 0 & 0 & 3 & 0 \\ 0 &  0 & 2 & -1 & 0 & 0 \\ 0 & 0 & -1 & 2 & 0 & 0 \\ 0 & 3 & 0 & 0 & 0 & 0 \\ 1 & 0 & 0 & 0 & 0 & 0 \end{psmallmatrix}$ using the pullback to reduce to $U \oplus U(3) \oplus A_1$. With respect to this Gram matrix, $L$ admits Borcherds products whose principal parts are given in Table \ref{tab1:U+U3+A2}.

\begin{table}[htbp]
\centering
\caption{Some Borcherds products for $U \oplus U(3) \oplus A_2$}\label{tab1:U+U3+A2}
\begin{tabular}{l*{3}{c}r}
\hline
Name & Weight & Principal part\\
\hline
$b_1$ & $3$ & $6 e_0 + q^{-1/3} e_{(0, 0, 1/3, 2/3, 0, 0)} + q^{-1/3} e_{(0, 0, 2/3, 1/3, 0, 0)}$ \\
$b_2$ & $3$ & $6 e_0 + q^{-1/3} e_{(0, 0, 1/3, 2/3, 2/3, 0)} + q^{-1/3} e_{(0, 0, 2/3, 1/3, 1/3, 0)}$ \\
$b_3$ & $3$ & $6 e_0 + q^{-1/3} e_{(0, 0, 1/3, 2/3, 1/3, 0)} + q^{-1/3} e_{(0, 0, 2/3, 1/3, 2/3, 0)}$ \\
$b_4$ & $3$ & $6 e_0 + q^{-1/3} e_{(0, 1/3, 0, 0, 1/3, 0)} + q^{-1/3} e_{(0, 2/3, 0, 0, 2/3, 0)}$ \\
$b_5$ & $3$ & $6 e_0 + q^{-1/3} e_{(0, 1/3, 2/3, 1/3, 0, 0)} + q^{-1/3} e_{(0, 2/3, 1/3, 2/3, 0, 0)}$ \\
$b_6$ & $3$ & $6 e_0 + q^{-1/3} e_{(0, 1/3, 1/3, 2/3, 0, 0)} + q^{-1/3} e_{(0, 2/3, 2/3, 1/3, 0, 0)}$ \\
$\Phi_{18}$ & $18$ & $36 e_0 + q^{-1} e_0$ \\
\hline
\end{tabular}
\end{table}

There is a natural embedding of $A_1$ in $A_2$ given by $x \mapsto (x, 0)$ with respect to the Gram matrices $\begin{pmatrix} 2 \end{pmatrix}$ and $\begin{psmallmatrix} 2 & -1 \\ -1 & 2 \end{psmallmatrix}$, and this induces an embedding of the modular variety associated to $U \oplus U(3) \oplus A_1$ in that of $U \oplus U(3) \oplus A_2$ whose image is exactly the divisor of $b_1$.

Finally we will want to consider additive lifts. The dimensions of modular forms for the Weil representation attached to $L$ have generating function $$\sum_{k=0}^{\infty} \mathrm{dim}\, M_k(\rho) t^k = \frac{1 + t + 3t^2 + 5t^3 + 3t^4 + 6t^5 + 3t^6 + 2t^7 + 3t^8}{(1 - t^4)(1 - t^6)}.$$ In particular, the dimensions of spaces of additive lifts of the smallest weights are $$\mathrm{dim}\, \mathrm{Maass}_1(L) = \mathrm{dim}\, \mathrm{Maass}_2(L) = 1, \; \mathrm{dim}\, \mathrm{Maass}_3(L) = 3, \; \mathrm{dim}\, \mathrm{Maass}_4(L) = 5.$$ We denote the additive lift of weight $1$ by $m_1$, and we let $m_4$ be any additive lift of weight $4$ that is not contained in $\mathrm{Maass}_1(L) \cdot \mathrm{Maass}_3(L)$.

\begin{theorem}\label{th:U+U3+A2_d} The six Borcherds products of weight $3$ span a three-dimensional space, and they satisfy the three-term relations $$b_1 + b_6 = b_5, \; b_4 + b_6 = b_2, \; b_4 + b_5 = b_3.$$ All of the weight three products are additive lifts. If $m_3^{(1)}, m_3^{(2)}, m_3^{(3)}$ are any linearly independent Maass lifts of weight three then the graded ring of modular forms for the discriminant kernel of $L$ is freely generated by $m_1, m_3^{(1)}, m_3^{(2)}, m_3^{(3)}, m_4$: $$M_*(\widetilde \Orth^+(L)) = \mathbb{C}[m_1, m_3^{(1)}, m_3^{(2)}, m_3^{(3)}, m_4],$$ and the Jacobian $J(m_1, m_3^{(1)}, m_3^{(2)}, m_3^{(3)}, m_4)$ is a constant multiple of the product $\Phi_{18}$.
\end{theorem}
\begin{proof} Let $\widetilde \Orth_r(L)$ denote the subgroup generated by reflections in $\widetilde \Orth^+(L)$, i.e. the reflections associated to the divisor of $\Phi_{18}$. Then all of the products of weight three have trivial character on $\widetilde \Orth_r(L)$. The pullback map in weight three $$M_3(\widetilde \Orth_r(L)) \longrightarrow M_3(\widetilde \Orth_r(U \oplus U(3) \oplus A_1))$$ has kernel spanned by the product $b_1$, so we obtain $$\mathrm{dim}\, M_3(\widetilde \Orth_r(L)) \le 1 + \mathrm{dim}\, M_3(\widetilde \Orth_r(U \oplus U(3) \oplus A_1)) = 4.$$  By examining the Fourier expansions of $m_1^3$, the six products of weight three, and the additive lifts we see that the dimension is exactly four, that the pullback in weight three is surjective, and that the Borcherds products span the three-dimensional space of additive lifts.

Since $m_4$ is not divisible by any of the weight three Maass lifts, its pullback to $U \oplus U(3) \oplus A_1$ is also not divisible by any of the weight three lifts; and since $m_1$ is not a multiple of $b_1$, its pullback to $U \oplus U(3) \oplus A_1$ is nonzero. In particular, the linear span of $m_1, m_3^{(1)}, m_3^{(2)}, m_3^{(3)}, m_4$ contains four forms whose pullback to $U \oplus U(3) \oplus A_1$ generate the graded ring of modular forms there, and one form with a simple zero along that divisor. Therefore the forms $m_1, m_3^{(1)}, m_3^{(2)}, m_3^{(3)}, m_4$ are algebraically independent. Using Theorems \ref{th:j} and \ref{th:j2} we find that their Jacobian $J$ equals $\Phi_{18}$ up to a constant multiple, that $\widetilde \Orth^+(L)$ is generated by reflections, and that $m_1, m_3^{(1)}, m_3^{(2)}, m_3^{(3)}, m_4$ generate the graded ring of modular forms.
\end{proof}

Finally we compute the graded ring of modular forms for the maximal reflection group. In this case the prospective Jacobian with simple zeros on all mirrors of reflections in $\Orth^+(L)$ is the product $b_1 b_2 b_3 b_4 b_5 b_6 \Phi_{18}$ of weight $36$. The lift $m_1$ constructed earlier is the theta lift of the Weil invariant $$f(\tau) = e_{v_1} - e_{-v_1} + e_{-v_2} - e_{v_2} + e_{v_3} - e_{-v_3} + e_{v_4} - e_{-v_4}$$ where, with respect to the Gram matrix above, we can take 
\begin{align*} 
&v_1 = (0,0,0,0,1/3,0),&  &v_2 = (0,2/3,0,0,0,0),& \\ 
&v_3 = (0,2/3,2/3,1/3,1/3,0),&  &v_4 = (0,2/3,1/3,2/3,1/3,0). &
\end{align*} 
The reflections associated to cosets in the divisors of $b_1,...,b_6$ each swap one pair of cosets in $(v_1+L,...,v_4+L)$ and preserve the remaining two, so $f$ is invariant under $\Orth_r(L)$ and therefore $m_1$ is modular under $\Orth_r(L)$. We produce a further modular form $m_9$ of weight $9$ for $\Orth_r(M)$ as the theta lift of $E_8(\tau) f(\tau) \in M_8(\rho)$, where $E_8(\tau) = 1 + 480 \sum_{n=1}^{\infty} \sigma_7(n) q^n$ is the usual Eisenstein series.

\begin{theorem} The graded ring of modular forms for $\Orth_r(L)$ is freely generated by forms of weights $1, 4, 6, 9, 12$: $$M_*(\Orth_r(L)) = \mathbb{C}[m_1, \mathcal{E}_4, \mathcal{E}_6, m_9, \mathcal{E}_{12}],$$ where $\mathcal{E}_k$ is the Eisenstein series of weight $k$. The Jacobian of the generators is a constant multiple of $J_0 := b_1 b_2 b_3 b_4 b_5 b_6 \Phi_{18}$.
\end{theorem}
\begin{proof} To prove that these forms are algebraically independent, one can express them in terms of the generators of Theorem \ref{th:U+U3+A2_d} or compute their Jacobian (in this case, showing that the Jacobian is not identically zero requires computing the forms involved to precision at least $O(q, s)^{14}$). Their Jacobian $J$ has weight $36$. By Theorems \ref{th:j} and \ref{th:j2} we conclude that $J = J_0$ up to a constant multiple and that these forms generate $M_*(\Orth_r(L))$.
\end{proof}

\begin{corollary} The graded ring of modular forms for $\Orth^+(L)$ is generated in weights $2,4,6,10,12,18$ with a single relation in weight $20$: $$M_*(\Orth^+(L)) = \mathbb{C}[m_1^2, \mathcal{E}_4, \mathcal{E}_6, m_1 m_9, \mathcal{E}_{12}, m_9^2] / R_{20},$$ where $R_{20}$ is the relation $m_1^2 \cdot m_9^2 = (m_1 m_9)^2$.
\end{corollary}
\begin{proof} This is because $m_1, m_9$ transform under $\Orth^+(L)$ with the same quadratic character by construction, and because the Eisenstein series are modular forms for $\Orth^+(L)$ with trivial character.
\end{proof}

\section{The \texorpdfstring{$U\oplus U(4)\oplus A_1$}{} lattice}\label{sec:6}

Let $L = U \oplus U(4) \oplus A_1$. We will determine the algebras of modular forms for the maximal reflection groups in $\Orth^+(L)$ and $\widetilde \Orth^+(L)$. These modular forms are essentially the same as modular forms for a level four subgroup of $\mathrm{Sp}_4(\mathbb{Z})$, and similar results have appeared in the work of Aoki--Ibukiyama \cite{AI}.

We first compute some relevant Borcherds products.

\begin{table}[htbp]
\centering
\caption{Some Borcherds products for $U \oplus U(4) \oplus A_1$}
\begin{tabular}{l*{3}{c}r}
\hline
Name & Weight & Principal part\\
\hline
$\theta$ & $1/2$ & $1 e_0 + q^{-1/4} e_{(0, 1/2, 1/2, 1/2, 0)}$ \\
$b_1$ & $2$ & $4 e_0 + q^{-1/4} e_{(0, 0, 1/2, 1/4, 0)} + q^{-1/4} e_{(0,0,1/2,3/4,0)}$ \\
$b_2$ & $2$ & $4 e_0 + q^{-1/4} e_{(0, 1/4, 1/2, 0, 0)} + q^{-1/4} e_{(0,3/4,1/2,0,0)}$ \\
$b_3$ & $2$ & $4 e_0 + q^{-1/4} e_{(0,1/4,0,1/4,0)} + q^{-1/4} e_{(0,3/4,0,3/4,0)}$ \\
$\Phi_{11}$ & $11$ & $22 e_0 + q^{-1} e_0$ \\
\hline
\end{tabular}
\end{table}

The space of vector-valued modular forms of weight $1/2$ for the Weil representation $\rho$ attached to $L$ is one-dimensional, spanned by the unary theta series 
\begin{align*} f(\tau) &= (1 + 2q + 2q^4 + 2q^9 + ...) (e_{v_1} - e_{-v_1} + e_{v_2} - e_{-v_2} + e_{v_3} - e_{v_3}) \\ &+ (2q^{1/4} + 2q^{9/4} + 2q^{25/4} + ...) (e_{w_1} - e_{-w_1} + e_{w_2} - e_{-w_2} + e_{w_3} - e_{-w_3})
\end{align*} 
where $v_i, w_i$ are the cosets 
\begin{align*} v_1 &= (0,0,0,1/4,0) &w_1 &= (0,3/4,0,1/4,0) \\ v_2 &= (0,3/4,0,0,0) &w_2 &= (0,1/2,1/2,1/4) \\ v_3 &= (0,3/4,1/2,1/4,0) &w_3 &= (0,3/4,1/2,1/2,0). 
\end{align*}
We let $m_1 = \Phi_f$ be the lift of $f$.

The vector-valued modular forms for the Weil representation attached to $\rho$ have the Hilbert series $$\sum_{k=0}^{\infty} \mathrm{dim}\, M_{k+1/2}(\rho) t^k = \frac{1 + 3t + 2t^2 + 4t^3 + 2t^4 + 3t^5 + t^6}{(1 - t^2)(1 - t^6)},$$ and by injectivity these correspond to the dimensions of the Maass subspaces. The first few dimensions are $$\mathrm{dim}\, \mathrm{Maass}_1(\Gamma) = 1, \; \mathrm{dim}\, \mathrm{Maass}_2(\Gamma) = \mathrm{dim}\, \mathrm{Maass}_3(\Gamma) = 3.$$ In particular, we can find two linearly independent lifts $m_2^{(1)}, m_2^{(2)}$ which are independent of $m_1^2$ (which turns out also to be a lift).  Finally we let $m_3 := \Phi_{\vartheta f}$ be the theta lift of the Serre derivative of $f$: $$\vartheta f(\tau) = \frac{1}{2\pi i} f'(\tau) - \frac{1}{24} E_2(\tau) f(\tau).$$

\begin{lemma} The Jacobian $J = J(m_1, m_2^{(1)}, m_2^{(2)}, m_3)$ is not identically zero.
\end{lemma}
\begin{proof} Whether the Jacobian is nonzero is independent of the choices of $m_2^{(1)}, m_2^{(2)}, m_3$ above, and it requires computing their Fourier expansions only to precision $O(q, s)^5$.
\end{proof}

\begin{theorem} The discriminant kernel of $L$ is generated by reflections. Its graded ring of modular forms is a free algebra generated in weights $1, 2, 2, 3$: $$M_*(\widetilde \Orth^+(L)) = \mathbb{C}[m_1, m_2^{(1)}, m_2^{(2)}, m_3].$$ The Jacobian $J$ of the generators is a nonzero multiple of $\Phi_{11}$.
\end{theorem}
\begin{proof} Theorem \ref{th:j} implies that $J = \Phi_{11}$ (up to a multiple), and Theorem \ref{th:j2} implies that the ring of modular forms for the reflection group $\widetilde \Orth_r(L)$ is $$M_*(\widetilde \Orth_r(L)) = \mathbb{C}[m_1, m_2^{(1)}, m_2^{(2)}, m_3].$$ Since all generators are theta lifts, they are modular under the a priori larger group $\widetilde \Orth^+(L)$, and we obtain $\widetilde \Orth_r(L) = \widetilde \Orth^+(L)$.
\end{proof}

\begin{remark} (i) Since $\widetilde \Orth^+(L)$ is generated by reflections, the singular-weight product $\theta$ transforms with a multiplier system of order two, and therefore $\theta^2 \in M_1(\widetilde \Orth^+(L))$. From the structure theorem we obtain $\theta^2 = m_1$. \\
(ii) The weight two products $b_1,b_2,b_3$ transform without character on $\widetilde \Orth^+(L)$, so together with $\theta^4$ they span the three-dimensional space $M_2(\widetilde \Orth^+(L))$.
\end{remark}

The ring of modular forms for the maximal reflection group $\Orth_r(L)$ can be determined by constructing the generators exactly as for the lattice $U\oplus U(3) \oplus A_1$. To obtain $\Orth_r(L)$ we have to add to $\widetilde \Orth^+(L)$ the reflections whose mirrors lie in the divisor of $b_1 b_2 b_3$. One can check that the theta function $f(\tau)$ that lifts to $m_1$ is invariant under all of these reflections, so $m_1$ is modular on $\Orth_r(L)$ without character. By construction the weight $3$ lift $m_3$ is also modular on $\Orth_r(L)$ without character.

\begin{theorem} The graded ring of modular forms for $\Orth_r(L)$ is freely generated in weights $1, 3, 4, 6$: $$M_*(\Orth_r(L)) = \mathbb{C}[m_1, m_3, \mathcal{E}_4, \mathcal{E}_6],$$ where $\mathcal{E}_4, \mathcal{E}_6$ are the Eisenstein series. The Jacobian of the generators is a constant multiple of $b_1 b_2 b_3 \Phi_{11}$.
\end{theorem}
\begin{proof} By computing Fourier expansions to at least precision $O(q, s)^8$ one can see that the Jacobian $J = J(m_1,m_3,\mathcal{E}_4,\mathcal{E}_6)$ is nonzero. $J$ has weight $17$ and by Theorem \ref{th:j} vanishes on all rational quadratic divisors that appear in the divisor of $J_0 := b_1 b_2 b_3 \Phi_{11}$. By Koecher's principle $J = J_0$ up to a nonzero multiple, and Theorem \ref{th:j2} implies that $m_1,m_3,\mathcal{E}_4,\mathcal{E}_6$ generate $M_*(\Orth_r(L))$.
\end{proof}

\begin{corollary} The ring of modular forms for $\Orth^+(L)$ is generated in weights $2, 4, 4, 6, 6$ with a relation in weight $8$: $$M_*(\Orth^+(L)) = \mathbb{C}[m_1^2, m_1 m_3, \mathcal{E}_4, m_3^2, \mathcal{E}_6] / R_8$$ where $R_8$ is the relation $m_1^2 \cdot m_3^2 = (m_1 m_3)^2$.
\end{corollary}

\section{The \texorpdfstring{$U\oplus S_8$}{} lattice}\label{sec:7}

As before, we let $S_8$ denote the lattice $\mathbb{Z}^3$ with Gram matrix $\begin{psmallmatrix} -2 & 1 & 1 \\ 1 & 2 & 1 \\ 1 & 1 & 2 \end{psmallmatrix}$ and genus symbol $8_3^{+1}$. Let $L = U \oplus S_8$. In this section we will show that the graded ring of modular forms for $\widetilde \Orth^+(L)$ is freely generated by additive lifts of weights $2, 4, 5, 6$. Computing the coefficients of a dual Eisenstein series yields Borcherds products with the following principal parts:

\begin{table}[htbp]
\centering
\caption{Some Borcherds products for $U \oplus S_8$}
\begin{tabular}{l*{3}{c}r}
\hline
Name & Weight & Principal part\\
\hline
$\psi$ & $5$ & $10e_0 + q^{-1/4} e_{(0, 1/4, 1/4, 1/4, 0)} + q^{-1/4} e_{(0, 3/4, 3/4, 3/4, 0)}$ \\
$\Phi_{20}$ & $20$ & $40 e_0 + q^{-1} e_0$ \\
\hline
\end{tabular}
\end{table}

The dimensions of modular forms for the Weil representation attached to $L$ have the generating function $$\sum_{k=0}^{\infty} \mathrm{dim}\, M_{k + 3/2}(\rho) t^k = \frac{1 + 2t^2 + t^3 + 2t^4 + t^5 + t^7}{(1 - t^4)(1 - t^6)}.$$ Since the additive lift is injective, we obtain the dimensions $$\mathrm{dim}\, \mathrm{Maass}_2(\Gamma) = 1, \; \mathrm{dim}\, \mathrm{Maass}_4(\Gamma) = 2, \; \mathrm{dim}\, \mathrm{Maass}_5(\Gamma) = 1, \; \mathrm{dim}\, \mathrm{Maass}_6(\Gamma) = 3.$$ In particular there is a unique normalized lift $m_2$ of weight two, and there are lifts of weights $4$ and $6$ that are not multiples of $m_2$. (Computing Fourier expansions shows that on can take the Eisenstein series $\mathcal{E}_4$ and $\mathcal{E}_6$ of weights $4$ and $6$.) Finally we obtain a one-dimensional space of additive lifts of weight $5$, which is spanned by the product $\psi$.

\begin{theorem} The modular groups $\widetilde \Orth^+(L)$ and $\Orth^+(L)$ are generated by reflections and their graded rings of modular forms are free algebras: 
\begin{align*}
M_*(\widetilde \Orth^+(L)) &= \mathbb{C}[m_2, \mathcal{E}_4, \psi, \mathcal{E}_6],\\
M_*(\Orth^+(L)) &= \mathbb{C}[m_2, \mathcal{E}_4, \mathcal{E}_6, \psi^2].
\end{align*}
\end{theorem}
\begin{proof} (i) The Jacobian $J = J(m_2, \mathcal{E}_4, \psi, \mathcal{E}_6)$ has weight $20$ and by computing Fourier expansions we see that it is not identically zero. Since all of the generators are additive lifts, Theorems \ref{th:j}, \ref{th:j2} and the argument in the previous sections shows that $\widetilde \Orth^+(L)$ is generated by reflections and that $M_*(\widetilde \Orth^+(L))$ is generated by $m_2, \mathcal{E}_4, \psi, \mathcal{E}_6$.

(ii) By (i), the Jacobian of the generators is a nonzero multiple of $\psi \Phi_{20}$. The input form whose additive lift is $m_2$ is invariant under $\Orth^+(L)$, so $m_2$ itself is modular under $\Orth^+(L)$. Since $\psi$ has at worst a quadratic character, it follows that all of the claimed generators are modular forms for $\Orth^+(L)$. Applying Theorems \ref{th:j}, \ref{th:j2} we see that $\Orth^+(L)$ is generated by reflections and that $M_*(\Orth^+(L))$ is generated by $m_2, \mathcal{E}_4, \mathcal{E}_6, \psi^2$.
\end{proof}

\section{The \texorpdfstring{$2U(3)\oplus A_1$}{}  and \texorpdfstring{$U \oplus U(2) \oplus A_1(2)$}{}  lattices}\label{sec:8}

In this section we discuss the algebras of modular forms for the lattices $L = 2U(3) \oplus A_1$ and $L = U \oplus U(2) \oplus A_1(2)$.  It was proved in \cite{WW20b} that for the first lattice $M_*(\widetilde\Orth^+(L))$ is freely generated by four forms of weight 1.  As mentioned in the introduction, $M_*(\widetilde\Orth_r(L))$ is not free for the second lattice. However, there is a free algebra of modular forms in weights $1,1,1,2$ for a reflection group slightly larger than $\widetilde\Orth_r(L)$ (see \cite{WW20b}).  It remains to determine the maximal reflection subgroups $\Orth_r(L)$ in $\Orth^+(L)$ and to determine the graded rings of modular forms for these groups.

\subsection{The \texorpdfstring{$2U(3)\oplus A_1$}{} lattice}

Let $L = 2U(3) \oplus A_1$. From Section 7 of \cite{WW20b} recall that $L$ admits $29$ holomorphic Borcherds products of weight $1$, of which one distinguished product denoted $\Delta_1$ has prinicipal part $2e_0 + q^{-1/4} e_v$ where $v$ has order two in $L'/L$, and that the ring of orthogonal modular forms for the discriminant kernel $M_*(\widetilde \Orth^+(L))$ is the polynomial algebra on any four linearly independent Borcherds products of weight $1$ other than $\Delta_1$. The form $\Delta_1$ turns out to play a special role in the structure of the algebra $M_*(\Orth^+(L))$. Recall that there are also twelve reflective Borcherds products of weight $1$ for this lattice, denoted $b_1,...,b_{12}$, each with a principal part of the form $2 e_0 + q^{-1/3} e_v + q^{-1/3} e_{-v}$ where $v$ has order $3$ in $L'/L$, and that there is a product $\Phi_7$ of weight $7$ whose principal part is $14e_0 + q^{-1} e_0$.

There is a two-dimensional space of vector-valued cusp forms of weight $3/2$ for the Weil representation attached to $L$, spanned by two forms $f_1, f_2$ of the form $$f(\tau) = \frac{\eta(2\tau)^5}{\eta(4\tau)^2} \sum_{\substack{x \in L'/L \\ Q(x) = -1/12 + \mathbb{Z}}} \varepsilon(x) e_x + 2 \frac{\eta(\tau)^2 \eta(4 \tau)^2}{\eta(2\tau)} \sum_{\substack{x \in L'/L \\ Q(x) = -1/3 + \mathbb{Z}}} \varepsilon(x) e_x$$ where $\varepsilon(x) \in \{-1, 0, 1\}$. Under the theta lift these are mapped to a one-dimensional space, spanned by the square $\Delta_1^2$. In particular, $\Delta_1^2$ transforms without character on the maximal reflection group $\Orth_r(L)$ (as it is a square) and with at worst a quadratic character on the full group $\Orth^+(L)$ (since this is true for all of the weight $3/2$ vector-valued cusp forms).

A direct calculation shows that the kernel of the additive lift contains no forms with rational Fourier coefficients; it is spanned by the form $g = f_1 + e^{2\pi i / 3} f_2$. For each $f_i$, there exists $c_i\in\CC$ such that $c_i\Delta_1^2=\Phi_f$. For any $\sigma \in \Orth^+(L)$, we have $\sigma(c\Delta_1^2)=\Phi_{\sigma(f)}$, which implies that $\sigma(f)\pm f$ has rational coefficients and lifts to zero, so $\sigma(f) \pm f = 0$. Therefore, each of $f_1, f_2$ are modular under $\Orth^+(L)$ with the same quadratic character $\chi$, and the kernel form $g$ is also modular under $\Orth^+(L)$ with the character $\chi$.

On the other hand, the Serre derivatives $\vartheta f_1, \vartheta f_2$ lift to a two-dimensional space of cusp forms of weight $4$ under the theta lift. Any form contained in this space transforms with the same character as $\Delta_1^2$ under $\Orth^+(L)$, since the Serre derivative commutes on input forms with the action of $\Orth^+(L)$. We let $m_4$ denote the additive theta lift of the Serre derivative of either $f_1$ or $f_2$ (the choice will not matter).

\begin{theorem} The group $\Orth^+(L)$ is generated by reflections. The graded algebra $M_*(\Orth^+(L))$ is freely generated in weights $2, 4, 4, 6$: $$M_*(\Orth^+(L)) = \mathbb{C}[\Delta_1^2, \mathcal{E}_4, m_4, \mathcal{E}_6],$$ where $\mathcal{E}_k$ is the Eisenstein series of weight $k$. 
\end{theorem}
\begin{proof} By computing its Fourier series we find that the Jacobian $J$ of the claimed generators is not identically zero, and has weight $19$. (Note that $J$ being nonzero is independent of the choice of $m_4$.) Since $J_0 := \Phi_7 \cdot \prod_{i=1}^{12} b_i$ has simple zeros exactly on the mirrors of reflections in $\Orth^+(L)$, Theorem \ref{th:j} and the Koecher principle imply that $J$ is a nonzero multiple of $J_0$. From Theorem \ref{th:j2} we conclude that $M_*(\Orth_r(L))$ is freely generated by $\Delta_1^2$, $\mathcal{E}_4$, $m_4$ and $\mathcal{E}_6$.

To see that $\Delta_1^2$ (and therefore $m_4$) has trivial character on $\Orth^+(L)$, one can argue as follows. $\Delta_1^4$ is itself a theta lift; namely, up to constant multiples, $\Delta_1^4 = \Phi_{\vartheta g}$ where $g$ is the theta kernel and $\vartheta g$ is the Serre derivative. (Proving this rigorously involves computing only two Fourier coefficients, due to the structure theorem for $M_*(\Orth_r(L))$. Then $\vartheta g$ transforms in the same way as $g$ under the action of $\Orth^+(L)$.  It follows that $\Delta_1^4$ transforms under $\Orth^+(L)$ with the same character as $\Delta_1^2$ and therefore that the character is trivial.
\end{proof}

\subsection{The \texorpdfstring{$U \oplus U(2) \oplus A_1(2)$}{} lattice}

Let $L = U \oplus U(2) \oplus A_1(2)$, with Gram matrix $\begin{psmallmatrix} 0 & 0 & 0 & 0 & 1\\ 0 & 0 & 0 & 2 & 0 \\ 0 & 0 & 4 & 0 & 0 \\ 0 & 2 & 0 & 0 & 0 \\ 1 & 0 & 0 & 0 & 0 \end{psmallmatrix}$. The lattice $L$ admits Borcherds products with principal parts as in Table \ref{tab:U+U(2)+A1(2)} below.

\begin{table}[htbp]
\centering
\caption{Some Borcherds products for $U \oplus U(2) \oplus A_1(2)$}\label{tab:U+U(2)+A1(2)}
\begin{tabular}{l*{3}{c}r}
\hline
Name & Weight & Principal part\\
\hline
$b_1$ & $1$ & $2 e_0 + q^{-1/8} e_{(0, 0, 1/4, 0, 0)} + q^{-1/8} e_{(0, 0, 3/4, 0, 0)}$ \\
$b_2$ & $1$ & $2 e_0 + q^{-1/8} e_{(0, 0, 1/4, 1/2, 0)} + q^{-1/8} e_{(0, 0, 3/4, 1/2, 0)}$ \\
$b_3$ & $1$ & $2 e_0 + q^{-1/8} e_{(0, 1/2, 1/4, 0, 0)} + q^{-1/8} e_{(0, 1/2, 3/4, 0, 0)}$ \\
$f_1$ & $4$ & $8 e_0 + q^{-1/2} e_{(0,0,1/2,1/2,0)}$ \\
$f_2$ & $4$ & $8 e_0 + q^{-1/2} e_{(0,1/2,1/2,0,0)}$ \\
$f_3$ & $4$ & $8 e_0 + q^{-1/2} e_{(0,1/2,0,1/2,0)}$ \\
$\Phi_4$ & $4$ & $8 e_0 + q^{-1} e_0$ \\
$\Psi_7$ & $7$ & $14 e_0 + q^{-1/2} e_{(0, 0, 1/2, 0, 0)}$ \\
\hline
\end{tabular}
\end{table}

Note that the quotient $\Psi_4 := \frac{\Psi_7}{b_1 b_2 b_3}$ is holomorphic even though some of the coefficients of the principal part of its input form are negative.

It was shown in \cite{WW20b} that the graded ring of modular forms for the subgroup $\Orth_1(L) \le \Orth^+(L)$ generated by reflections associated to the divisor of $\Phi_4 \Psi_4$ is freely generated by $b_1, b_2, b_3$ and an additive lift $m_2$ of weight two. (Note that the ring of modular forms for maximal reflection group in $\widetilde \Orth^+(L)$ cannot be freely generated, since the prospective Jacobian $\Phi_4$ is not a cusp form.)

\begin{theorem} The group $\Orth^+(L)$ is generated by reflections. The graded ring of modular forms for $\Orth^+(L)$ is the free algebra on generators of weights $2, 4, 6, 8$: $$M_*(\Orth^+(L)) = \mathbb{C}[m_2, \mathcal{E}_4, \mathcal{E}_6, \mathcal{E}_8].$$ The Jacobian of the generators is a constant multiple of $f_1 f_2 f_3 \Phi_4 \Psi_7 = b_1 b_2 b_3 f_1 f_2 f_3 \Phi_4 \Psi_4$.
\end{theorem}
\begin{proof} The Eisenstein series $\mathcal{E}_4, \mathcal{E}_6, \mathcal{E}_8$ are clearly modular forms without character on $\Orth^+(L)$. The input function to $m_2$ in the theta lift can be constructed as a linear combination of (non-holomorphic) Eisenstein series of weight $3/2$: $$E_{3/2, 0} - \sum_{\substack{x \in L'/L \\Q(x) = 0,\, x \ne 0}} E_{3/2, x},$$ which is easily seen to be invariant under all automorphisms of the discriminant form, such that $m_2$ is also modular without character on $\Orth^+(L)$. By computing Fourier expansions we find that the Jacobian is not identically zero and is of weight $23$. Since $J_0 := f_1 f_2 f_3 \Phi_4 \Psi_7$ has weight $23$ and a simple zero on all mirrors of reflections in $\Orth^+(L)$, Theorem \ref{th:j} and the Koecher principle imply that $J = J_0$ up to a nonzero constant, and Theorem \ref{th:j2} implies that $\Orth^+(L)$ is generated by reflections and that $m_2, \mathcal{E}_4, \mathcal{E}_6, \mathcal{E}_8$ generate $M_*(\Orth^+(L))$.
\end{proof}

\subsection{The remaining cases}\label{sec:8.3} To complete the proof of Theorem \ref{Maintheorem}, we have to prove the last five cases in Table \ref{Maintab1}. For any even lattice $L$ and any positive integer $a$ we have
$$
\Orth^+(L(a))=\Orth^+(L)=\Orth^+(L')=\Orth^+(L'(a)).
$$
Using this trick, we have the following isomorphisms:
\begin{align*}
&\Orth^+(2U(2)\oplus A_1(2))=\Orth^+(2U\oplus A_1)=\Orth^+(2U(4)\oplus A_1),\\
&\Orth^+(U(4)\oplus U(2)\oplus A_1)=\Orth^+(U\oplus U(2)\oplus A_1),\\
&\Orth^+(2U(3)\oplus A_2)=\Orth^+(2U\oplus A_2).
\end{align*}
From \cite[Lemma 6.1]{GN18}, we deduce that 
$$
\Orth^+(2U(2)\oplus A_1)=\Orth^+(2U\oplus A_1(2)).
$$
This reduces these algebras to four cases which are already known.

\section{The \texorpdfstring{$U \oplus U(2) \oplus 2A_1$}{} lattice} \label{sec:9}

The lattice $L = U \oplus U(2) \oplus 2A_1$ with Gram matrix $\begin{psmallmatrix} 0 & 0 & 0 & 0 & 0 & 1 \\ 0 & 0 & 0 & 0 & 2 & 0 \\ 0 & 0 & 2 & 0 & 0 & 0 \\ 0 & 0 & 0 & 2 & 0 & 0 \\ 0 & 2 & 0 & 0 & 0 & 0 \\ 1 & 0 & 0 & 0 & 0 & 0 \end{psmallmatrix}$ is not simple, as the dual Weil representation attached to $L$ admits a two-dimensional space of cusp forms of weight $3$, spanned by the forms $$f_1(\tau) = \eta(\tau)^6 (e_{(0,0,1/2,0,1/2,0)} + e_{(0,1/2,0,1/2,0,0)} - e_{(0,0,0,1/2,1/2,0)} - e_{(0,1/2,1/2,0,0,0)})$$ and $$f_2(\tau) = \eta(\tau)^6 (e_{(0,0,1/2,0,0,0)} + e_{(0,1/2,1/2,0,0,0)} - e_{(0,0,0,1/2,0,0)} -e_{(0,1/2,0,1/2,0,0)}).$$ Nevertheless, it behaves similarly to simple lattices in many ways, since the obstruction principle \cite{Bo2} implies that every Heegner divisor which is invariant under swapping the two $A_1$ components occurs as the divisor of a Borcherds product. By computing the Fourier coefficients of the Eisenstein series one can obtain the Borcherds products in Table \ref{tab:U+U(2)+2A1}.

\begin{table}[htbp]
\centering
\caption{Some Borcherds products for $U \oplus U(2) \oplus 2A_1$}\label{tab:U+U(2)+2A1}
\begin{tabular}{l*{3}{c}r}
\hline
Name & Weight & Principal part\\
\hline
$b_1$ & $4$ & $8 e_0 + q^{-1/2} e_{(0, 0, 1/2, 1/2, 1/2, 0)}$ \\
$b_2$ & $4$ & $8 e_0 + q^{-1/2} e_{(0, 1/2, 1/2, 1/2, 0, 0)}$ \\
$b_3$ & $4$ & $8 e_0 + q^{-1/2} e_{(0, 1/2, 0, 0,1/2, 0)}$ \\
$\psi$ & $6$ & $12 e_0 + q^{-1/2} e_{(0, 0, 1/2, 1/2, 0, 0)}$ \\
$\Phi_{18}$ & $18$ & $36 e_0 + q^{-1} e_0$ \\
\hline
\end{tabular}
\end{table}

The involution $\varepsilon(x_1,x_2,x_3,x_4,x_5,x_6) = (x_1,x_2,x_4,x_3,x_5,x_6)$ on $L$ induces an involution (also labelled $\varepsilon$) of $\mathbb{C}[L'/L]$ and splits the spaces of vector-valued modular forms for $\rho_L$ naturally into eigenspaces: $$M_k(\rho) = M_{k, +}(\rho) \oplus M_{k, -}(\rho) \; \text{where} \; M_{k, \pm}(\rho) = \{f \in M_k(\rho): \; \varepsilon \circ f = \pm f\}.$$ Note that any element of $M_{k,-}(\rho)$ is necessarily a cusp form because all isotropic cosets of $L'/L$ are invariant under $\varepsilon$. Let $S_{k, \pm}(\rho) = S_k(\rho) \cap M_{k, \pm}(\rho)$ denote the spaces of cusp forms. Using the Riemann--Roch formula one can compute $$\mathrm{dim}\, M_1(\rho) = \mathrm{dim}\, M_{1, +}(\rho) = 1;$$ $$\mathrm{dim}\, M_3(\rho) = \mathrm{dim}\, M_{3, +}(\rho) = 4;$$ $$\mathrm{dim}\, M_5(\rho) = 6, \; \mathrm{dim}\, S_{5, +}(\rho) = \mathrm{dim}\, S_{5, -}(\rho) = 1.$$ In fact, $M_1(\rho)$ is spanned by the vector-valued Eisenstein series $E_1(\tau)$ of weight one, which lifts to the Eisenstein series $\mathcal{E}_2$ of weight two. It will turn out that $(\mathcal{E}_2)^2$ is an additive theta lift of weight four. Choose any three theta lifts $m_{4,1}, m_{4,2}, m_{4,3}$ of weight four such that $\{\mathcal{E}_2^2, m_{4,1}, m_{4,2}, m_{4,3}\}$ is linearly independent, and let $m_6$ be the theta lift of the normalized form in $S_{5,+}(\rho)$.

\begin{theorem} Let $\Gamma = \langle \widetilde \Orth^+(L), \varepsilon \rangle$ be the group generated by the discriminant kernel $\widetilde \Orth^+(L)$ and the involution $\varepsilon$ that swaps the two $A_1$ components in $L$. Then $\Gamma$ is generated by reflections and $M_*(\Gamma)$ is a free algebra generated by forms of weights $2,4,4,4,6$: $$M_*(\Gamma) = \mathbb{C}[\mathcal{E}_2, m_{4,1}, m_{4,2}, m_{4,3}, m_6].$$
\end{theorem}
\begin{proof} The divisor of $\psi \Phi_{18}$ consists exactly of mirrors of reflections contained in $\Gamma$. Since all of the generators are additive lifts of forms invariant under $\varepsilon$, they are modular forms for $\Gamma$ without character.
By Theorem \ref{th:j} the Jacobian $J = J(\mathcal{E}_2, m_{4,1}, m_{4,2}, m_{4,3}, m_6)$ is divisible by $\psi \Phi_{18}$; since $J$ and $\psi \Phi_{18}$ each have weight $24$, it follows that $J / (\psi \Phi_{18})$ is a constant. Computing $J$ with at least the first seven Fourier--Jacobi coefficients of these generators shows that the constant is nonzero. The structure of $M_*(\Gamma)$ and the fact that $\Gamma$ is generated by reflections associated to the divisor of $\psi \Phi_{18}$ then follow from Theorem \ref{th:j2}.
\end{proof}

\begin{remark} The Borcherds products of weight $4$ are additive lifts and transform under $\Gamma$ without character. However, they cannot be used as generators in place of the additive lifts $m_{4,i}$ because they are linearly dependent.
\end{remark}

\begin{remark} Any modular form that transforms with eigenvalue $-1$ under $\varepsilon$ has a forced zero on the divisor of $\psi$. Moreover, $\psi$ itself satisfies $\varepsilon(\psi) = -\psi$ because it has odd order on the mirror corresponding to $\varepsilon$. It follows that the graded ring $M_*(\widetilde \Orth^+(L))$ is generated by $\mathcal{E}_2, m_{4, 1}, m_{4, 2}, m_{4, 3}, m_6, \psi$ modulo a single relation in weight $12$ of the form $\psi^2 = P(\mathcal{E}_2, m_{4,1}, m_{4,2}, m_{4,3}, m_6)$. In particular, $M_*(\widetilde \Orth^+(L))$ is not a free algebra.
\end{remark}

We also obtain the ring of modular forms for $\Orth^+(L)$ with this argument.

\begin{theorem}\label{th:U+U(2)+2A1}
The group $\Orth^+(L)$ is generated by reflections. The graded ring $M_*(\Orth^+(L))$ is freely generated by the Eisenstein series of weights $2,4,6,8,12$: $$M_*(\Orth^+(L)) = \mathbb{C}[\mathcal{E}_2, \mathcal{E}_4, \mathcal{E}_6, \mathcal{E}_8, \mathcal{E}_{12}].$$
\end{theorem}
\begin{proof} By construction the Eisenstein series are modular forms for the full group $\Orth^+(L)$. The Jacobian $J = J(\mathcal{E}_2, \mathcal{E}_4, \mathcal{E}_6, \mathcal{E}_8, \mathcal{E}_{12})$ vanishes on the mirrors of all reflections in $\Orth^+(L)$ and in particular is divisible by $J_0 := b_1b_2b_3\psi\Phi_{18}$. Since $J$ and $J_0$ both have weight $36$, it follows that $J / J_0$ is a constant. Computing the Fourier--Jacobi expansions of the Eisenstein series to precision at least $9$ shows that $J$ is not identically zero. It follows that $$M_*(\Orth^+(L)) =  \mathbb{C}[\mathcal{E}_2, \mathcal{E}_4, \mathcal{E}_6, \mathcal{E}_8, \mathcal{E}_{12}]$$ 
and $\Orth^+(L)$ is generated by reflections corresponding to $J$. 
\end{proof}

\begin{remark}
There are some other free algebras of modular forms associated to non-simple lattices. For the convenience of the reader, we list all known such algebras for maximal reflection groups contained in the full orthogonal groups in Table \ref{tab:non-simple}. 
\begin{table}[htbp]
\caption{Free algebras of modular forms on $\Orth_r(L)$ for non-simple lattices.}\label{tab:non-simple}
\renewcommand\arraystretch{1.1}
\noindent\[
\begin{array}{cccc}
\hline
\# & L  & \text{Weights of generators} & \text{Proof}\\
\hline
1& 2U \oplus 2A_1 & 4, 6, 8, 10, 12 & \text{\cite{WW20a}}\\
2& 2U \oplus 3A_1 & 4, 6, 6, 8, 10, 12 & \\
3& 2U \oplus 4A_1 & 4, 4, 6, 6, 8, 10, 12 & \\
\hline
4& 2U \oplus 2A_1(2) & 2, 4, 4, 6, 6 & \text{\cite{Wan20a}}\\
5& 2U \oplus A_2(2) &  4, 4, 6, 6, 6 & \\
6& 2U \oplus A_2(3) &  2, 2, 4, 4, 6 & \\
7& 2U \oplus A_3(2) & 2, 4, 4, 6, 6, 6 & \\
8& 2U \oplus D_4(2) & 4, 4, 4, 6, 6, 6, 6 & \\
\hline
9& U\oplus U(2) \oplus 2A_1 & 2, 4, 6, 8, 12 & \text{Theorem \ref{th:U+U(2)+2A1}} \\
\hline
\end{array} 
\]
\end{table}

It was proved in \cite{WW20b} that the graded algebra of modular forms for the group generated by $\widetilde \Orth_r(2U(2)\oplus 2A_1)$ and the swapping of two $A_1$ components is freely generated by five forms of weight 2. However, we do not obtain a new free algebra for $\Orth_r(2U(2)\oplus 2A_1)$ because by \cite[Lemma 6.1]{GN18} we have
$$
\Orth^+(2U(2)\oplus 2A_1)=\Orth^+(2U\oplus 2A_1).
$$
\end{remark}

\begin{remark}
Similarly to our previous work \cite{Wan20a, WW20b}, there are free algebras of modular forms for some reflection groups smaller than $\widetilde \Orth_r$ and for some reflection groups between $\widetilde \Orth_r$ and $\Orth_r$, which can be computed using the argument of this paper. We leave this task to the reader.

Conjecture 5.2 in \cite{Wan} states that if $M_*(\Gamma)$ is a free algebra for a finite index subgroup $\Gamma$ of $\Orth^+(L)$ then $M_*(\Gamma_1)$ is also free for any other reflection subgroup $\Gamma_1$ satisfying  $\Gamma< \Gamma_1 < \Orth^+(L)$. All examples in the current paper and in our previous work support this conjecture.
\end{remark}

\bigskip

\noindent
\textbf{Acknowledgements} 
H. Wang is grateful to Max Planck Institute for Mathematics in Bonn for its hospitality and financial support. B. Williams is supported by a fellowship of the LOEWE research group Uniformized Structures in Arithmetic and Geometry. The computations in this paper used the Sage computer algebra system \cite{sagemath}.

\bibliographystyle{plainnat}
\bibliofont
\bibliography{refs}

\end{document}